\documentclass[dvipdfmx]{article}

\usepackage{amsmath,amssymb,amsthm,pifont,graphics,tikz,proof}

\title{The persistence principle over weak interpretability logic}
\author{Sohei Iwata\footnote{Email: siwata@dpc.agu.ac.jp}
\footnote{Division of Liberal Arts and Sciences, Aichi-Gakuin University}, 
Taishi Kurahashi\footnote{Email: kurahashi@people.kobe-u.ac.jp}
\footnote{Graduate School of System Informatics, Kobe University, Japan}
and Yuya Okawa\footnote{Email: math.y.okawa@gmail.com}
\footnote{Graduate School of Science and Engineering, Chiba University, Japan}
}

\date{}

\theoremstyle{plain}
\newtheorem{thm}{Theorem}[section]
\newtheorem{lem}[thm]{Lemma}
\newtheorem{prop}[thm]{Proposition}
\newtheorem{cor}[thm]{Corollary}
\newtheorem{fact}[thm]{Fact}
\newtheorem{prob}[thm]{Problem}
\newtheorem*{cl}{Claim}

\theoremstyle{definition}
\newtheorem{defn}[thm]{Definition}

\newcommand{\PA}{\mathrm{PA}}
\newcommand{\N}{\mathbb{N}}

\newcommand{\K}{\mathbf{K}}
\newcommand{\GL}{\mathbf{GL}}
\newcommand{\FGL}{\mathbf{GL} \otimes \mathbf{GL}}
\newcommand{\GLK}{\mathbf{GL} \otimes \mathbf{K}}
\newcommand{\IL}{\mathbf{IL}}
\newcommand{\ILM}{\mathbf{ILM}}
\newcommand{\ILP}{\mathbf{ILP}}
\newcommand{\PP}{\mathbf{P}}

\newcommand{\G}[1]{\mathbf{L#1}}
\newcommand{\J}[1]{\mathbf{J#1}}
\newcommand{\R}[1]{\mathbf{R#1}}

\newcommand{\var}{\mathrm{var}}

\newcommand{\ILms}{{\IL}^-_{\mathrm{seq}}}
\newcommand{\ILmPs}{{\IL^-(\PP)}_{\mathrm{seq}}}
\newcommand{\WL}{\mathrm{wl}}
\newcommand{\WR}{\mathrm{wr}}
\newcommand{\NL}{\neg\mathrm{l}}
\newcommand{\NR}{\neg\mathrm{r}}

\newcommand{\seq}[1]{\langle#1\rangle}

\newcommand{\Prf}{\mathrm{Prf}}
\newcommand{\PR}{\mathrm{Pr}}
\newcommand{\Con}{\mathrm{Con}}
\newcommand{\gn}[1]{\ulcorner#1\urcorner}
\newcommand{\FTT}[1]{f_{\tau_0, \tau_1}(#1)}
\newcommand{\FRS}[1]{f_{\tau_R, \tau_S}(#1)}

\begin{document}

\maketitle

\begin{abstract}
We focus on the persistence principle over weak interpretability logic.
Our object of study is the logic obtained by adding the persistence principle to weak interpretability logic from several perspectives. 
Firstly, we prove that this logic enjoys a weak version of the fixed point property. 
Secondly, we introduce a system of sequent calculus and prove the cut-elimination theorem for it. 
As a consequence, we prove that the logic enjoys the Craig interpolation property. 
Thirdly, we show that the logic is the natural basis of a generalization of simplified Veltman semantics, and prove that it has the finite frame property with respect to that semantics. 
Finally, we prove that it is sound and complete with respect to some appropriate arithmetical semantics.
\end{abstract}

\section{Introduction}

The notion of interpretability between theories of mathematics has been studied using the framework of modal logic.
Visser~\cite{Vis88,Vis90} introduced the logic $\IL$ which is an extension of the logic $\GL$ of provability in the language of $\GL$ augmented by the binary modal operator $\rhd$. 
In this framework, the formula $A \rhd B$ is intended as ``$T + A$ is interpretable in $T + B$'' for some suitable theory $T$. 
Visser also introduced extensions $\ILM$ and $\ILP$ of $\IL$, and it has been shown that these logics are complete with respect to such arithmetical interpretations. 
In particular, Visser~\cite{Vis90} proved that the logic $\ILP$ which is obtained from $\IL$ by adding the persistence principle $\PP$: $A \rhd B \to \Box(A \rhd B)$ as an axiom is arithmetically complete for finitely axiomatizable appropriate theories $T$. 

De Jongh and Sambin~\cite{Sam76} independently proved that $\GL$ has the fixed point property (FPP), that is, for any modal formula $A(p)$ in the language of $\GL$, if each occurrence of $p$ in $A(p)$ is in some scope of $\Box$, then there exists a modal formula $F$ such that $\var(F) \subseteq \var(A(p)) \setminus \{p\}$ and $F \leftrightarrow A(F)$ is provable in $\GL$, where $\var(F)$ is the set of all propositional variables contained in $F$. 
This is a modal counterpart of the Fixed Point Lemma for theories of arithmetic used in a standard proof of G\"odel's incompleteness theorems. 
De Jongh and Visser~\cite{DeJVis91} proved that the de Jongh--Sambin fixed point theorem can be extended to the language of $\IL$, and that $\IL$ has FPP. 

The authors have analyzed the question of which sublogics of $\IL$ are sufficient to enjoy FPP.  
First,~\cite[Definition 2.1]{KO21} introduced the logic $\IL^-$ as a basis for semantical investigations of sublogics of $\IL$, and proved the completeness and the incompleteness of several logics between $\IL^-$ and $\IL$ with respect to some relational semantics. 
Then,~\cite{IKO20} studied FPP for these sublogics. 
In particular, it was proved that the fixed point theorem holds for the sublogic $\IL^-(\J{2}_+, \J{5})$, and that it does not hold for several sublogics of $\IL$. 
Moreover,~\cite{IKO20} introduced a weaker version $\ell$FPP of the fixed point property, and proved that the sublogic $\IL^-(\J{4}, \J{5})$ has $\ell$FPP. 
Furthermore,~\cite{Okawa} proved that every element of the infinite descending sequence $\langle \IL^-(\J{2}_+, \J{5}^n) \rangle_{n \geq 1}$ (resp.~$\langle \IL^-(\J{4}, \J{5}^n) \rangle_{n \geq 1}$) has FPP (resp.~$\ell$FPP). 

These logics are not the only extensions of $\IL^-$ having FPP or $\ell$FPP.
De Jongh and Visser~\cite{DeJVis91} gave a simple proof that a sublogic $\mathbf{SR_1}$ of $\ILP$ has FPP. 
Then, in our context, it immediately follows that $\IL^-(\J{4}_+, \PP)$ also has FPP, although $\IL^-(\J{4}_+)$ does not have FPP as shown in~\cite{IKO20}.  
This observation suggests that the persistence principle $\PP$ may be logically well-suited to sublogics of $\IL$. 

Our motivation for the present paper is to analyze the behavior and effects of the persistence principle $\PP$ over the logic $\IL^-$. 
Our main research focus is the logic $\IL^-(\PP)$ obtained by adding $\PP$ into $\IL^-$. 
In fact, although it was proved in~\cite{IKO20} that $\IL^-$ does not have $\ell$FPP, in Section~\ref{Sec_Pre}, we show that $\IL^-(\PP)$ enjoys $\ell$FPP. 
Thus, it can be seen that the principle $\PP$ has a beneficial effect on $\IL^-$ from the viewpoint of the fixed point theorem.
We study the logic $\IL^-(\PP)$ from three aspects: proof theoretic, relational semantic, and arithmetical semantic aspects in Sections~\ref{Sec_CE},~\ref{Sec_SV}, and~\ref{Sec_AC}, respectively. 

In Section~\ref{Sec_CE}, we investigate proof theoretic aspects of $\IL^-(\PP)$. 
The systems of sequent calculus for the logics $\IL$ and $\ILP$ have been introduced and studied by Sasaki \cite{Sas02_1,Sas02_2,Sas03}. 
In the present paper, we introduce the systems $\ILms$ and $\ILmPs$ of sequent calculus for the logics $\IL^-$ and $\IL^-(\PP)$, respectively.  
We prove that the cut-elimination theorem holds for $\ILmPs$. 
Then, we prove that Maehara's method can be applied to the systems $\ILms$ and $\ILmPs$, that is, every sequent having a cut-free proof has a Craig interpolant. 
Therefore, we conclude that $\IL^-(\PP)$ has the Craig interpolation property. 
On the other hand, it was shown in~\cite{IKO20} that $\IL^-$ does not have the Craig interpolation property, and hence we obtain that the cut-elimination theorem does not hold for $\ILms$.
From these results, we can see that $\PP$ is a principle that behaves proof theoretically well for $\IL^-$.

In Section~\ref{Sec_SV}, we investigate relational semantics of $\IL^-(\PP)$. 
The logics $\IL$ and $\ILP$ have relational semantics which are extensions of Kripke semantics. 
A triple $(W, R, \{S_w\}_{w \in W})$ is called a \textit{Veltman frame} if $(W, R)$ is a $\GL$-frame and for each $w \in W$, $S_w$ is a transitive and reflexive binary relation on $R[w] = \{x \in W \mid w R x\}$ satisfying $(\forall x, y \in W) (w R x \ \&\ x R y \Rightarrow x S_w y)$.  
De Jongh and Veltman~\cite{DeJVel90} proved the completeness theorems of $\IL$ and $\ILP$ with respect to Veltman semantics. 
The treatment of the family $\{S_w\}_{w \in W}$ in Veltman semantics is somewhat complicated, and then a simplified semantics was introduced by Visser~\cite{Vis88}. 
A triple $(W, R, S)$ is called a \textit{simplified Veltman frame} or a \textit{Visser frame} if $(W, R)$ is a $\GL$-frame and $S$ is a transitive and reflexive binary relation on $W$ satisfying $(\forall x, y \in W) (x R y \Rightarrow x S y)$. 
Then, Visser proved that $\IL$ and $\ILP$ are also complete with respect to simplified Veltman semantics. 

The notion of Veltman frames can be generalized. 
We say that a triple $(W, R, \{S_w\}_{w \in W})$ is an $\IL^-\!$-frame if $(W, R)$ is a $\GL$-frame and for each $w \in W$, $S_w \subseteq R[w] \times W$.  
Then,~\cite{KO21} proved that $\IL^-$ is characterized by the class of all finite $\IL^-\!$-frames. 
In this context, it seems natural to generalize the notion of simplified Veltman semantics as well. 
Let us consider triples $(W, R, S)$ such that $(W, R)$ is a $\GL$-frame and $S$ is simply a binary relation on $W$. 
Then, unlike the case of Veltman semantics, in the present paper, we prove that the logic $\IL^-(\PP)$ is valid in all such frames. 
Moreover, we prove that $\IL^-(\PP)$ is characterized by the class of all finite such frames. 
Therefore, $\IL^-(\PP)$ is the natural basis of a generalization of simplified Veltman semantics. 
This result also shows that simplified Veltman semantics is a useful device for analyzing logics containing the principle $\PP$. 
Also, as an application of that result, we show that $\IL^-(\PP)$ is faithfully embeddable into several extensions of the fusion $\GLK$ of $\GL$ and $\K$. 

Finally, in Section~\ref{Sec_AC}, we investigate an arithmetical semantics of $\IL^-(\PP)$. 
Inspired from our embedding result of $\IL^-(\PP)$ into bimodal logics, we introduce appropriate arithmetical semantics for $\IL^-(\PP)$ based on Fefermanian provability predicates satisfying L\"ob's derivability conditions (See Visser~\cite{Vis21}). 
Then, by tracing the proof of the main result in~\cite{Kur18_2}, we prove that $\IL^-(\PP)$ is sound and complete with respect to such arithmetical semantics.

\section{Sublogics of the interpretability logic and fixed point properties}\label{Sec_Pre}

In this section, we introduce the logics $\IL$ and $\IL^-$. 
For several logics related to them, we summarize known results on the fixed point theorem and the Craig interpolation theorem. 
We also show that the logic $\IL^-(\PP)$ has $\ell$FPP.

The language $\mathcal{L}(\rhd)$ consists of countably many propositional variables $p, q, r, \ldots$, the logical constant $\bot$, Boolean connectives $\neg, \land, \lor$, and $\to$, the unary modal operator $\Box$, and the binary modal operator $\rhd$. 
The logical constant $\top$ and the modal operator $\Diamond$ are introduced as abbreviations for $\neg \bot$ and $\neg \Box \neg$, respectively. 
The axioms of the logic $\IL$ in the language $\mathcal{L}(\rhd)$ are as follows: 
\begin{description}
	\item [L1] All tautologies in $\mathcal{L}(\rhd)$; 
	\item [L2] $\Box(A \to B) \to (\Box A \to \Box B)$; 
	\item [L3] $\Box(\Box A \to A) \to \Box A$; 
	\item [J1] $\Box(A \to B) \to A \rhd B$; 
	\item [J2] $(A \rhd B) \land (B \rhd C) \to A \rhd C$; 
	\item [J3] $(A \rhd C) \land (B \rhd C) \to (A \lor B) \rhd C$; 
	\item [J4] $A \rhd B \to (\Diamond A \to \Diamond B)$;
	\item [J5] $\Diamond A \rhd A$. 
\end{description}
The inference rules of $\IL$ are Modus Ponens $\dfrac{A \to B \quad A}{B}$ and Necessitation $\dfrac{A}{\Box A}$. 

Several sublogics were introduced in~\cite{KO21}, and the basis for them is the logic $\IL^-$. 
The axioms of $\IL^-$ are $\G{1}$, $\G{2}$, $\G{3}$, $\J{3}$, and $\J{6}$: $\Box A \leftrightarrow (\neg A) \rhd \bot$. 
The inference rules of $\IL^-$ are Modus Ponens, Necessitation, $\R{1}$ $\dfrac{A \to B}{C \rhd A \to C \rhd B}$, and $\R{2}$ $\dfrac{A \to B}{B \rhd C \to A \rhd C}$. 
Let $\IL^-(\Sigma_1, \ldots, \Sigma_n)$ be the logic obtained from $\IL^-$ by adding the schemes $\Sigma_1, \ldots, \Sigma_n$ as axioms. 
The following schemes have been analyzed in previous studies. (See~\cite{DeJVis91,KO21,Okawa,Vis88}): 
\begin{description}
	\item [J2$_+$] $(A \rhd (B \lor C)) \land (B \rhd C) \to A \rhd C$; 
	\item [J4$_+$] $\Box (A \to B) \to (C \rhd A \to C \rhd B)$; 
	\item [J5$^n$] $\Diamond^n A \rhd A$ ($n \geq 1$); 
	\item [E1] $\Box(A \leftrightarrow B) \to (C \rhd A \leftrightarrow C \rhd B)$; 
	\item [E2] $\Box(A \leftrightarrow B) \to (A \rhd C \leftrightarrow B \rhd C)$; 
	\item [P] $A \rhd B \to \Box(A \rhd B)$. 
\end{description}

For these axiom schemes, it has been shown that the following holds.

\begin{fact}[cf.~\cite{KO21,Okawa}]\label{IL-fact}\leavevmode
\begin{enumerate}
	\item $\IL^- \vdash \Box(A \to B) \to (B \rhd C \to A \rhd C)$. 
	Hence, $\IL^- \vdash \mathbf{E2}$.
	\item $\IL^-(\J{4}_+) \vdash \J{4} \land \mathbf{E1}$. 
	\item $\IL^-(\J{2}_+) \vdash \J{2} \land \J{4}_+$. 
	\item $\IL^-(\J{2}) \vdash \J{4}$. 
	\item $\IL^-(\J{5}^n) \vdash \J{5}^{n+1}$ for each $n \geq 1$. 
	\item The logics $\IL$, $\IL^-(\J{1}, \J{2}, \J{5})$, and $\IL^-(\J{1}, \J{2}_+, \J{5})$ are deductively equivalent. 
\end{enumerate}

\end{fact}

For any $\mathcal{L}(\rhd)$-formula $A$, let $\var(A)$ denote the set of all propositional variables appearing in $A$. 

\begin{defn}[FPP]\leavevmode
\begin{enumerate}
	\item We say that a propositional variable $p$ is \textit{modalized} in an $\mathcal{L}(\rhd)$-formula $A$ if all occurrences of $p$ in $A$ are in scope of some modal operator $\Box$ or $\rhd$. 
	\item The logic $L$ is said to have the \textit{fixed point property (FPP)} if for any propositional variable $p$ and any $\mathcal{L}(\rhd)$-formula $A(p)$ in which $p$ is modalized, there exists an $\mathcal{L}(\rhd)$-formula $F$ such that $\var(F) \subseteq \var(A(p)) \setminus \{p\}$ and $L \vdash F \leftrightarrow A(F)$. 
Such a formula $F$ is called a \textit{fixed point} of $A(p)$.
\end{enumerate}
\end{defn}

De Jongh and Visser~\cite[p.~47]{DeJVis91} proved that the logic $\IL$ has FPP. 
This result was slightly improved by showing that the fixed point theorem holds without the axiom scheme $\J{1}$, namely, the logic $\IL^-(\J{2}_+, \J{5})$ also has FPP (\cite[Corollary 5.1 and Theorem 5.2]{IKO20}). 
Moreover, it was proved that for each $n \geq 2$, the logic $\IL^-(\J{2}_+, \J{5}^n)$ has FPP (\cite[Theorem 3.21.1]{Okawa}). 
Hence, each element of the infinite decreasing sequence $\langle \IL^-(\J{2}_+, \J{5}^n) \rangle_{n \geq 1}$ of sublogics of $\IL^-(\J{2}_+, \J{5})$ has FPP, and the intersection of all the logics of the sequence is exactly $\IL^-(\J{2}_+)$ (\cite[Proposition 3.11]{Okawa}). 
It was also shown that the logic $\IL^-(\J{2}_+)$ does not have FPP (\cite[Corollary 6.2]{IKO20}). 

De Jongh and Visser also proved the fixed point theorem in another way. 
They proved that the logic containing $\G{1}$, $\G{2}$, $\G{3}$, $\PP$, $\mathbf{E1}$, and $\mathbf{E2}$, and is closed under Modus Ponens and Necessitation has FPP (\cite[Corollary 2.5.(a)]{DeJVis91}). 
From their result, it immediately follows that the logic $\IL^-(\J{4}_+, \PP)$ has FPP. 

For some technical reasons, the following weaker version of FPP was introduced in~\cite[Definition 3.8]{IKO20}: 

\begin{defn}[$\ell$FPP]\leavevmode
\begin{enumerate}
	\item We say that a propositional variable $p$ is \textit{left-modalized} in an $\mathcal{L}(\rhd)$-formula $A$ if $p$ is modalized in $A$ and for any subformula $B \rhd C$ of $A$, we have $p \notin \var(C)$. 

	\item The logic $L$ is said to have \textit{$\ell$FPP} if for any propositional variable $p$ and any $\mathcal{L}(\rhd)$-formula $A(p)$ in which $p$ is left-modalized, there exists an $\mathcal{L}(\rhd)$-formula $F$ such that $\var(F) \subseteq \var(A(p)) \setminus \{p\}$ and $L \vdash F \leftrightarrow A(F)$.
\end{enumerate}
\end{defn}

Then, it was proved that the logic $\IL^-(\J{4}, \J{5})$ has $\ell$FPP (\cite[Theorem 5.9]{IKO20}). 
Moreover, as in the case of FPP, it was proved that for each $n \geq 2$, the logic $\IL^-(\J{4}, \J{5}^n)$ also has $\ell$FPP (\cite[Theorem 3.21.2]{Okawa}). 

The fact that the logic $\IL^-(\J{4}_+, \PP)$ has FPP suggests that some weaker logics are expected to have $\ell$FPP. 
Indeed, we prove: 

\begin{thm}\label{lFPP}
The logic $\IL^-(\PP)$ has $\ell$FPP. 
\end{thm}

In the proof, we use the following fact. 

\begin{fact}[{\cite[Proposition~3.6]{IKO20}}]\label{Sbsti}
Let $A, B$, and $C(p)$ be $\mathcal{L}(\rhd)$-formulas. 
\begin{enumerate}
	\item If $p \notin \var(E)$ for every subformula $D \rhd E$ of $C(p)$, then 
\[
	\IL^- \vdash (A \leftrightarrow B) \land \Box (A \leftrightarrow B) \to \bigl(C(A) \leftrightarrow C(B)\bigr).
\] 
	\item If $p$ is left-modalized in $C(p)$, then 
\[
	\IL^- \vdash \Box(A \leftrightarrow B) \to \bigl(C(A) \leftrightarrow C(B)\bigr).
\]
\end{enumerate}
\end{fact}

\begin{proof}[Proof of Theorem~\ref{lFPP}]
As in a usual proof of the fixed point theorem, it suffices to find a fixed point of $A(p) \rhd B$ in which $p$ is left-modalized. 
Let $F$ be $A(\top) \rhd B$, and we prove that $F$ is a fixed point of $A(p) \rhd B$ in $\IL^-(\PP)$. 

Obviously, $\var(F) \subseteq \var\bigl(A(p) \rhd B\bigr) \setminus \{p\}$. 
We show $\IL^-(\PP) \vdash F \leftrightarrow A(F) \rhd B$. 
Since $\IL^- \vdash F \to (\top \leftrightarrow F)$, we have $\IL^- \vdash \Box F \to \Box(\top \leftrightarrow F)$. 
By Fact~\ref{Sbsti}.2, we have $\IL^- \vdash \Box F \to \bigl(A(\top) \rhd B \leftrightarrow A(F) \rhd B\bigr)$. 
Hence, 
\begin{eqnarray}\label{eq2}
\IL^- \vdash \Box F \to (F \leftrightarrow A(F) \rhd B). 
\end{eqnarray}
Therefore, $\IL^-(\PP) \vdash F \to A(F) \rhd B$ because $\IL^-(\PP)$ proves $F \to \Box F$. 

By the right-to-left direction of (\ref{eq2}),  $\IL^- \vdash A(F) \rhd B \to (\Box F \to F)$, and then $\IL^- \vdash \Box(A(F) \rhd B) \to \Box (\Box F \to F)$. 
Hence, $\IL^- \vdash \Box(A(F) \rhd B) \to \Box F$. 
We get $\IL^-(\PP) \vdash A(F) \rhd B \to \Box  F$. 
By (\ref{eq2}) again, $\IL^-(\PP) \vdash A(F) \rhd B \to F$. 

We conclude $\IL^-(\PP) \vdash F \leftrightarrow A(F) \rhd B$. 
\end{proof}

Notice that our proof of Theorem~\ref{lFPP} is essentially the same as de Jongh and Visser's proof of FPP for $\IL^-(\J{4}_+, \PP)$. 
Then, our fixed points obtained in our proof are the same as ones given by de Jongh and Visser. 

The properties FPP and $\ell$FPP are closely related to the Craig interpolation property. 

\begin{defn}[CIP]
We say that a logic $L$ has the \textit{Craig interpolation property (CIP)} if for any $\mathcal{L}(\rhd)$-formulas $A$ and $B$, if $L \vdash A \to B$, then there exists an $\mathcal{L}(\rhd)$-formula $C$ such that $\var(C) \subseteq \var(A) \cap \var(B)$, $L \vdash A \to C$, and $L \vdash C \to B$. 
\end{defn}
Areces, Hoogland, and de Jongh~\cite[Theorem 1]{AHD01} proved that $\IL$ has CIP. 
Moreover, it was proved that $\IL^-(\J{2}_+, \J{5})$ has CIP, but some sublogics including $\IL^-$ do not have CIP (See \cite{IKO20}). 
The failure of CIP for some logics was derived by showing the failure of FPP for these logics through the following fact. 

\begin{fact}[{\cite[Lemmas 3.10 and 3.11]{IKO20}}]\label{CIPFPP}
Let $L$ be any extension of $\IL^-$ that is closed under substituting a formula for a propositional variable. 
\begin{enumerate}
	\item If $L$ has CIP, then $L$ has $\ell$FPP. 
	\item If $L$ is an extension of $\IL^-(\J{4}_+)$ and has CIP, then $L$ has FPP. 
\end{enumerate}
\end{fact}

In the next section, we prove that $\IL^-(\PP)$ has CIP. 
Then, we also obtain an alternative proof of Theorem~\ref{lFPP}.

\section{Proof theoretic aspects of weak interpretability logic with persistence}\label{Sec_CE}

In this section, we investigate proof theoretic aspects of $\IL^-(\PP)$. 
This section consists of three subsections. 
In the first subsection, we introduce the systems $\ILms$ and $\ILmPs$ of sequent calculi and prove that they exactly correspond to $\IL^-$ and $\IL^-(\PP)$, respectively. 
The second subsection is devoted to proving the cut-elimination theorem for $\ILmPs$. 
In the last subsection, we show that Maehara's method \cite{Mae61} for proving the existence of interpolants for sequents with a cut-free proof can be applied to $\ILms$ and $\ILmPs$. 
As a consequence, we obtain that $\IL^-(\PP)$ enjoys CIP, but the cut-elimination theorem does not hold for $\ILms$. 

\subsection{Sequent calculi of weak interpretability logics}

As an exception in the present paper, the language of our systems of sequent calculi does not contain the modal symbol $\Box$, and we assume that $\Box A$ is an abbreviation of $(\neg A )\rhd \bot$. 
This corresponds to the scheme $\J{6}$. 
Throughout this section, capital Greek letters $\Gamma, \Delta, \Sigma, \ldots$ always denote finite \textit{sets} of $\mathcal{L}(\rhd)$-formulas. 
As usual, for each $\mathcal{L}(\rhd)$-formula $A$, the expressions $\Gamma, \Delta$ and $\Gamma, A$ denote $\Gamma \cup \Delta$ and $\Gamma \cup \{A\}$, respectively. 
Let $\Box \Gamma$ be the set $\{\Box A \mid A \in \Gamma\}$. 
A \textit{sequent} is an expression of the form $\Gamma \Rightarrow \Delta$.

\begin{defn}[Sequent calculus $\ILms$]
The system $\ILms$ is defined by the following initial sequents and inference rules:
\begin{description}
\item[Initial sequents]
\begin{align*}
A \Rightarrow A & &  \bot \Rightarrow
\end{align*}
\item[Logical rules]
\begin{align*}
\infer[(\WL)]{\Gamma, A \Rightarrow \Delta}{\Gamma \Rightarrow \Delta} & &
\infer[(\WR)]{\Gamma \Rightarrow A, \Delta}{\Gamma \Rightarrow \Delta} \\
\infer[(\NL)]{\Gamma, \neg A \Rightarrow \Delta}{\Gamma \Rightarrow A, \Delta} & &
\infer[(\NR)]{\Gamma \Rightarrow \neg A, \Delta}{\Gamma, A \Rightarrow \Delta} \\
\infer[(\land\mathrm{l})]{\Gamma, A_1 \land A_2 \Rightarrow \Delta}{\Gamma, A_i \Rightarrow \Delta \ (i =1,2)} & &
\infer[(\land\mathrm{r})]{\Gamma \Rightarrow A \land B, \Delta}
	{\Gamma \Rightarrow A, \Delta & \Gamma \Rightarrow B, \Delta} \\
\infer[(\lor\mathrm{l})]{\Gamma, A \lor B \Rightarrow \Delta}
	{\Gamma, A \Rightarrow \Delta & \Gamma, B \Rightarrow \Delta} & &
\infer[(\lor\mathrm{r})]{\Gamma \Rightarrow A_1 \lor A_2, \Delta}{\Gamma \Rightarrow A_i, \Delta \ (i=1,2)} \\
\infer[(\to \mathrm{l})]{\Gamma, A \to B \Rightarrow \Delta}
	{\Gamma \Rightarrow A, \Delta & \Gamma, B \Rightarrow \Delta} & &
\infer[(\to \mathrm{r})]{\Gamma \Rightarrow A \to B, \Delta}{\Gamma, A \Rightarrow B, \Delta} 
\end{align*}
\item[Cut rule]
\[
\infer[(\mathrm{cut})]{\Gamma, \Sigma \Rightarrow \Delta, \Pi}{\Gamma \Rightarrow \Delta, A & A, \Sigma \Rightarrow \Pi} 
\]
\item[Modal rules]
\[
\infer[(\Box)]{\Box \Gamma \Rightarrow \Box A}{\Box \Gamma, \Gamma, \Box A \Rightarrow A}
\quad
\infer[(\rhd)]{\{ X_i \rhd Y_i \mid i < n \} \Rightarrow A \rhd B}{ A \Rightarrow \{ X_i \mid i < n\} & \langle Y_i \Rightarrow B \rangle_{i < n} }
\]
\end{description}
In the rule $(\mathrm{cut})$, the formula $A$ is said to be the \textit{cut-formula}. 
In the rule $(\rhd)$, $n$ is some natural number, which is possibly $0$. 
Also, in the rule $(\rhd)$, $\langle Y_i \Rightarrow B \rangle_{i < n}$ indicates that all $n$ elements of this sequence are upper sequents of the rule, allowing for duplication. 
\end{defn}

Note that our system of sequent calculus is set-based, and hence the structural rules are only $(\WL)$ and $(\WR)$. 
If there is no room for misreading, repeated trivial applications of propositional rules are sometimes abbreviated by a double line. 
Notice that the rule $(\Box)$ is known as the rule for $\GL$, that is, it corresponds to $\G{1}$, $\G{2}$, $\G{3}$, and Necessitation. 
Also, the following derivation shows that the rule $\dfrac{\Gamma \Rightarrow A}{\Box \Gamma \Rightarrow \Box A}$ for the smallest normal modal logic $\K$ is admissible only from the rule $(\rhd)$: For $\Gamma = \{X_0, \ldots, X_{n-1}\}$, 
\[
	\infer[(\rhd)]{\{(\neg X_i) \rhd \bot \mid i < n\} \Rightarrow (\neg A) \rhd \bot}
	{\infer=[(\neg \mathrm{l})\text{ and }(\neg\mathrm{r})\text{s}]{\neg A \Rightarrow \{\neg X_i \mid i < n\}}
	{\{X_i \mid i < n\} \Rightarrow A}
	&
	\overbrace{\bot \Rightarrow \bot}^{n}
	}.
\]

For the sake of simplicity, we sometimes abbreviate the set $\{X_i \mid i < n\}$ of formulas as $\{X_i\}$, but we believe that this will not cause any confusions.

We prove that the system $\ILms$ is equivalent to $\IL^-$. 

\begin{prop}\label{IL^-equiv}
Let $A$ be any $\mathcal{L}(\rhd)$-formula.  
\begin{enumerate}
	\item If $\IL^- \vdash A$, then $\ILms \vdash (\Rightarrow A)$. 
	\item If $\ILms \vdash \Gamma \Rightarrow \Delta$, then $\IL^- \vdash \bigwedge \Gamma \to \bigvee \Delta$. 
\end{enumerate}
\end{prop}
\begin{proof}
1. We prove the proposition by induction on the length of proofs in $\IL^-$. 
Notice that the axiom scheme $\J{6}$ is built in $\ILms$. 
The following derivation shows that $\ILms$ proves the sequents corresponding to the axiom scheme $\J{3}$. 
\[
\infer[(\rhd): n = 2]{A \rhd C, B \rhd C \Rightarrow (A \lor B) \rhd C}
		{A \lor B \Rightarrow A, B
	&
	C \Rightarrow C
	&
	C \Rightarrow C
	}. 
\]
Then, it suffices to show that the rules $\R{1}$ and $\R{2}$ are admissible in $\ILms$. 

\begin{description}
	\item [R1]
\[
	\infer[(\rhd): n  = 1]{C \rhd A \Rightarrow C \rhd B}
		{C \Rightarrow C
	&
	A \Rightarrow B
	}.
\]

	\item [R2]
\[
	\infer[(\rhd): n = 1]{B \rhd C \Rightarrow A \rhd C}
	{A \Rightarrow B
	&
	C \Rightarrow C
	}.
\]
\end{description}

2. It is sufficient to show that the rule $(\rhd)$ is admissible in $\IL^-$. 
Suppose that $\IL^- \vdash A \to \bigvee \{X_i\}$ and $\IL^- \vdash Y_j \to B$ for each $j < n$. 
Since $\IL^- \vdash Y_i \to \bigvee \{Y_j\}$ for each $i < n$, we have $\IL^- \vdash X_i \rhd Y_i \to X_i \rhd \bigvee \{Y_j\}$ by the rule $\R{1}$. 
Then, $\IL^- \vdash \bigwedge \{X_i \rhd Y_i \} \to \bigwedge \{X_i \rhd \bigvee \{Y_j\} \}$. 
By the $n-1$ times applications of $\J{3}$, we obtain 
\begin{equation}\label{seq1}
	\IL^- \vdash \bigwedge \{X_i \rhd Y_i\} \to \left( \bigvee \{X_i\} \right) \rhd \bigvee \{Y_j\}. 
\end{equation}
Since $\IL^- \vdash A \to \bigvee \{X_i\}$, by the rule $\R{2}$, we have
\begin{equation}\label{seq2}
	\IL^- \vdash \left( \bigvee \{X_i\} \right) \rhd \bigvee\{Y_j\} \to A \rhd \bigvee \{Y_j\}. 
\end{equation}
Since $\IL^- \vdash \bigvee \{Y_j\} \to B$, we obtain
\[
	\IL^- \vdash A \rhd \bigvee \{Y_j\} \to A \rhd B
\]
by $\R{1}$. 
By combining this with (\ref{seq1}) and (\ref{seq2}), we conclude
\[
	\IL^- \vdash \bigwedge \{X_i \rhd Y_i\} \to A \rhd B. \tag*{\mbox{\qedhere}}
\] 
\end{proof}

We will prove the failure of the cut-elimination theorem for $\ILms$ at the end of this section. 
In contrast, when we add the persistence principle $\PP$ to $\ILms$, the cut-elimination method works well. 
We introduce $\ILmPs$ the system of sequent calculus for $\IL^-(\PP)$. 
For the sake of simplicity, throughout this section, $\Omega$ and $\Theta$, with subscripts in some cases, always denote finite sets of $\mathcal{L}(\rhd)$-formulas of the forms $C \rhd D$. 
Then, $\IL^-(\PP) \vdash \bigwedge \Omega \to \Box \bigwedge \Omega$. 

\begin{defn}[Sequent calculus $\ILmPs$]
The system $\ILmPs$ is obtained from the system $\ILms$ by replacing the two modal rules $(\Box)$ and $(\rhd)$ with the following single modal rule $(\rhd_\PP)$:
\begin{center}
$\infer[(\rhd_\PP)]{\Omega, \{ X_i \rhd Y_i \mid i < n\} \Rightarrow A \rhd B}
{\Omega, \{ X_i \rhd Y_i \mid i < n\}, A \rhd B, A \Rightarrow \{ X_i \mid i < n\}
&
\langle Y_i \Rightarrow B \rangle_{i < n}
}.$
\end{center}
In the rule, the set $\Omega$ is possibly empty. 
The elements of $\{X_i \rhd Y_i \mid i < n\}$ are said to be the \textit{principal formulas} of the rule.  
Also, the formula $A \rhd B$ is said to be the \textit{diagonal formula} of the rule.
\end{defn}

The following proposition states the equivalence between $\IL^-(\PP)$ and $\ILmPs$.

\begin{prop}
Let $A$ be any $\mathcal{L}(\rhd)$-formula.  
\begin{enumerate}
	\item If $\IL^-(\PP) \vdash A$, then $\ILmPs \vdash (\Rightarrow A)$. 
	\item If $\ILmPs \vdash \Gamma \Rightarrow \Delta$, then $\IL^-(\PP) \vdash \bigwedge \Gamma \to \bigvee \Delta$. 
\end{enumerate}
\end{prop}

\begin{proof}
1. First, the following derivations show that rules $(\Box)$ and $(\rhd)$ in $\ILms$ are admissible in $\ILmPs$. 

\begin{description}
	\item [$(\Box)$] For $\Gamma = \{X_0, \ldots, X_{n-1}\}$, 
\begin{equation*}
\infer[(\rhd_\PP): \Omega = \varnothing]{\{(\neg X_i) \rhd \bot \mid i < n\} \Rightarrow (\neg A) \rhd \bot}
	{\infer=[(\neg\text{l})\text{ and }(\neg\text{r})\text{s}]{\{(\neg X_i) \rhd \bot \mid i < n\}, (\neg A) \rhd \bot, \neg A \Rightarrow \{\neg X_i \mid i < n\}}
		{\{(\neg X_i) \rhd \bot \mid i < n\}, \{X_i \mid i < n\}, (\neg A) \rhd \bot \Rightarrow A}
		&
		\overbrace{\bot \Rightarrow \bot}^{n}
	}.
\end{equation*}

	\item [$(\rhd)$] 
\begin{equation*}
\infer[(\rhd_\PP): \Omega = \varnothing]{\{X_i \rhd Y_i \mid i < n\} \Rightarrow A \rhd B}
	{\infer=[(\WL)\text{s}]{\{X_i \rhd Y_i \mid i < n\}, A \rhd B, A \Rightarrow \{X_i \mid i < n\}}
		{A\Rightarrow \{X_i \mid i < n\}}
		&
		\langle Y_i \Rightarrow B \rangle_{i < n}
	}.
\end{equation*}
\end{description}

Then, Proposition~\ref{IL^-equiv}.1 shows that the sequents corresponding to theorems of $\IL^-$ are provable in $\ILmPs$. 
Then, it suffices to show that the sequents corresponding to $\PP$ are provable in $\ILmPs$. 
\begin{align*}
\infer[(\rhd_\PP): n = 0\ \text{and}\ \Omega = \{A \rhd B\}]{A \rhd B \Rightarrow (\neg(A \rhd B)) \rhd \bot}
	{\infer[(\NL)]{A \rhd B, (\neg(A \rhd B)) \rhd \bot, \neg(A \rhd B) \Rightarrow}
		{\infer[(\WL)]{A \rhd B, (\neg(A \rhd B)) \rhd \bot \Rightarrow A \rhd B}
			{A \rhd B \Rightarrow A \rhd B}
		}
	}.
\end{align*}

2. It suffices to show that the rule $(\rhd_\PP)$ is admissible in $\IL^-(\PP)$. 
Assume that $\IL^-(\PP)$ proves
\begin{equation*}
\bigwedge \Omega \land \bigwedge \{ X_i \rhd Y_i \} \land (A \rhd B) \land A \to \bigvee \{ X_i \}
\end{equation*}
and $Y_j \to B$ for each $j < n$. 
Then, 
\begin{align}\label{seq3}
\IL^-(\PP) & \vdash \bigwedge \Omega \land \bigwedge \{ X_i \rhd Y_i \} \land (A \rhd B) \to \left(A \to \bigvee \{ X_i \} \right), \nonumber\\
& \vdash \bigwedge \Box \Omega \land \bigwedge \Box \{ X_i \rhd Y_i \} \land \Box (A \rhd B) \to \Box \left(A \to \bigvee \{ X_i \}\right), \nonumber\\
& \vdash \bigwedge \Omega \land \bigwedge \{ X_i \rhd Y_i \} \land \Box (A \rhd B) \to \Box \left(A \to \bigvee \{ X_i \}\right).
\end{align}
Also, as in the proof of Proposition~\ref{IL^-equiv}.2, we have $\IL^- \vdash \bigwedge \{ X_i \rhd Y_i \} \to \left( \bigvee \{ X_i \} \right) \rhd \bigvee \{ Y_i \}$. 
By Fact \ref{IL-fact}.1, 
\begin{equation*}
\IL^- \vdash \Box \left(A \to \bigvee \{ X_i \} \right) \to \left( \left( \bigvee \{ X_i \} \right) \rhd \bigvee \{ Y_i \} \to A \rhd \bigvee \{ Y_i \} \right). 
\end{equation*}
Thus, we obtain
\begin{equation*}
\IL^- \vdash \bigwedge \{ X_i \rhd Y_i \} \to \left( \Box \left(A \to \bigvee \{ X_i \} \right) \to A \rhd \bigvee \{ Y_i \} \right).
\end{equation*}
From this with (\ref{seq3}), 
\begin{equation*}
\IL^-(\PP) \vdash \bigwedge \Omega \land \bigwedge \{ X_i \rhd Y_i \} \land \Box (A \rhd B) \to A \rhd \bigvee \{ Y_i \}.
\end{equation*}
Since $\IL^-(\PP) \vdash \bigvee \{ Y_i \} \to B$, by the rule $\R{1}$, $\IL^-(\PP) \vdash  A \rhd \bigvee \{ Y_i \} \to A \rhd B$. 
Hence, we obtain
\begin{equation*}
\IL^-(\PP) \vdash \bigwedge \Omega \land \bigwedge \{ X_i \rhd Y_i \} \land \Box (A \rhd B) \to A \rhd B.
\end{equation*}
Finally,
\begin{align}\label{seq4}
\IL^-(\PP) & \vdash \bigwedge \Omega \land \bigwedge \{ X_i \rhd Y_i \} \to ( \Box (A \rhd B) \to A \rhd B), \\
& \vdash  \bigwedge \Box \Omega \land \bigwedge \Box \{ X_i \rhd Y_i \} \to \Box ( \Box (A \rhd B) \to A \rhd B), \nonumber \\
& \vdash \bigwedge \Omega \land \bigwedge \{ X_i \rhd Y_i \} \to \Box (A \rhd B), \nonumber \\
& \vdash \bigwedge \Omega \land \bigwedge \{ X_i \rhd Y_i \} \to A \rhd B. \tag{By (\ref{seq4})} \nonumber
\end{align}
\end{proof}

\subsection{Proof of the cut-elimination theorem}

In this subsection, we prove the cut-elimination theorem for $\ILmPs$. 
That is, if a sequent $\Gamma \Rightarrow \Delta$ is derivable in $\ILmPs$, then it is also derivable without using $(\mathrm{cut})$s.

We prove the cut-elimination theorem by double induction on degree and height. 
For an instance of $(\mathrm{cut})$ of the form
$$\infer[(\mathrm{cut})]{\Gamma, \Sigma \Rightarrow \Delta, \Pi}{\deduce{\Gamma \Rightarrow \Delta, A}{\pi} & \deduce{A, \Sigma \Rightarrow \Pi}{\sigma}}$$
where $\pi$ and $\sigma$ are cut-free subproofs, we define the degree and height of this application of the rule $(\mathrm{cut})$ as follows:
\begin{itemize}
\item The \textit{degree} of the $(\mathrm{cut})$ is the number of occurrences of connectives and operators in the cut-formula $A$.
\item The \textit{height} of the $(\mathrm{cut})$ is the sum of the maximum lengths of the subproofs $\pi$ and $\sigma$.
\end{itemize}

The usual cut-elimination procedure searches a topmost instance of the $(\mathrm{cut})$ rule and converts it into a cut-free proof of the same conclusion. 
The conversion method is given by primary induction on the degree and secondary induction on the height. 
If the cut-formula of such a topmost instance is one of the propositional variables, $\bot$, $\neg A$, $A \lor B$, $A \land B$, or $A \to B$, then it is verified in the standard way. 
Thus, we focus only on the rule $(\mathrm{cut})$ where the cut-formula is of the form $A \rhd B$. 
Even in this case, if at least one of the last applications of $\pi$ and $\sigma$ is a logical rule, then the secondary induction hypothesis goes well. 
Therefore, the only essential case is that the last applications of $\pi$ and $\sigma$ are both $(\rhd_\PP)$.
Let us consider the following derivation:
\begin{equation*}
\infer[(\mathrm{cut})]
{\Omega_1, \Omega_2, \{ X_i \rhd Y_i \}, \{ Z_j \rhd W_j \} \Rightarrow C \rhd D}
{\deduce{\Omega_1, \{ X_i \rhd Y_i \} \Rightarrow A \rhd B}{\pi}
&
\deduce{\Omega_2, \{ Z_j \rhd W_j \}, A \rhd B \Rightarrow C \rhd D}{\sigma}
},
\end{equation*}
where $\pi$ and $\sigma$ are cut-free subproofs whose last applications are both $(\rhd_\PP)$. 
If the cut-formula $A \rhd B$ is non-principal in the last application of $\sigma$, then the following argument shows that a cut-free proof of the lower sequent is easily obtained from the induction hypothesis. 
Suppose that $\sigma$ is of the form:
\[ 
\infer[(\rhd_\PP)]{A \rhd B, \Omega_2, \{ Z_j \rhd W_j \} \Rightarrow C \rhd D}
	{\deduce{A \rhd B, \Omega_2, \{ Z_j \rhd W_j \}, C \rhd D, C \Rightarrow \{ Z_j \}}{\sigma_{\mathrm{L}}}
	&
	\langle \deduce{W_j \Rightarrow D}{\sigma_j} \rangle
	}.
\]

From $\pi$ and $\sigma_{\mathrm{L}}$, applying $(\mathrm{cut})$ with the cut-formula $A \rhd B$, we obtain a proof of
\[
\left(
\begin{array}{ccc}
\Omega_1, & \{ X_i \rhd Y_i \}, & C \rhd D, \\
\Omega_2, & \{ Z_j \rhd W_j \}, & C
\end{array}
\Rightarrow \{ Z_j \}
\right).
\]
By the induction hypothesis, this application of $(\mathrm{cut})$ is eliminable because its height is smaller than that of the original $(\mathrm{cut})$. 
Therefore, we obtain a cut-free proof of the above sequent.  
The required cut-free proof is obtained as follows: 
\[
\infer[(\rhd_\PP)]{\Omega_1, \Omega_2, \{ X_i \rhd Y_i \}, \{ Z_j \rhd W_j \} \Rightarrow C \rhd D}
	{\begin{array}{ccc}
	\Omega_1, & \{ X_i \rhd Y_i \}, & C \rhd D, \\
	\Omega_2, & \{ Z_j \rhd W_j \}, & C
	\end{array}
	\Rightarrow 
	\{ Z_j \}
	&
	\deduce{\langle W_j \Rightarrow D \rangle}{\sigma_j}
	}.
\]
Notice that the formulas in $\{ X_i \rhd Y_i \}$ are non-principal in this application of $(\rhd_\PP)$.

So, it suffices to prove the following main proposition.
\begin{prop}\label{mainP}
Suppose that the following proof is given:
\begin{equation*}
\infer[(\mathrm{cut})]
{\Omega_1, \Omega_2, \{ X_i \rhd Y_i \}, \{ Z_j \rhd W_j \} \Rightarrow C \rhd D}
{\deduce{\Omega_1, \{ X_i \rhd Y_i \} \Rightarrow A \rhd B}{\pi}
&
\deduce{\Omega_2, \{ Z_j \rhd W_j \}, A \rhd B \Rightarrow C \rhd D}{\sigma}
},
\end{equation*}
where $\pi$ and $\sigma$ are cut-free subproofs of the following forms respectively:
\begin{itemize}
\item $\pi: \infer[(\rhd_\PP)]
	{\Omega_1, \{ X_i \rhd Y_i \} \Rightarrow A \rhd B}
	{\deduce{\Omega_1, \{ X_i \rhd Y_i \}, A \rhd B, A \Rightarrow \{ X_i \}}{\pi_{\mathrm{L}}}
	&
	\langle \deduce{Y_i \Rightarrow B}{\pi_i} \rangle
	}$,
\item 
$\sigma: \infer[(\rhd_\PP)]
	{\Omega_2, \{ Z_j \rhd W_j \}, A \rhd B \Rightarrow C \rhd D}
	{\deduce{\Omega_2, \{ Z_j \rhd W_j \}, A \rhd B, C \rhd D, C \Rightarrow \{ Z_j \}, A}{\sigma_{\mathrm{L}}}
	&
	\langle \deduce{W_j \Rightarrow D}{\sigma_j} \rangle
	&
	\deduce{B \Rightarrow D}{\sigma_B}}.$
\end{itemize}
Then, there exists a cut-free proof of $(\Omega_1, \Omega_2, \{ X_i \rhd Y_i \}, \{ Z_j \rhd W_j \} \Rightarrow C \rhd D)$.
\end{prop}

The remainder of this subsection is devoted to proving Proposition \ref{mainP}. 
For this purpose, we prove a lemma along the lines of a modification of Borga's argument \cite{Bor83}.
We assume that we already have the fixed cut-free proofs $\pi$, $\sigma$, $\pi_{\mathrm{L}}$, $\pi_i (i<n)$, $\sigma_{\mathrm{L}}$, $\sigma_j (j < m)$, and $\sigma_B$ as indicated in the statement of Proposition~\ref{mainP}. 
We say that a proof is an \textit{$(A, B)^\star$-proof} if every cut-formula of the rule $(\mathrm{cut})$ occurring in the proof is either $A$ or $B$.
For the subproof $\pi_{\mathrm{L}}$, we inductively define \textit{$(A \rhd B)$-explicit} sequents of $\pi_{\mathrm{L}}$ as follows:
\begin{enumerate}
\item The root sequent $(\Omega_1, \{ X_i \rhd Y_i \}, A \rhd B, A \Rightarrow \{X_i\})$ is $(A \rhd B)$-explicit;
\item If $\dfrac{S_1 \quad \cdots \quad S_n}{S_0}$ is an application of some rule in $\pi_{\mathrm{L}}$, $S_0$ is $(A \rhd B)$-explicit, and $A \rhd B$ occurs in the antecedent of the sequent $S_i$, then $S_i$ is also $(A \rhd B)$-explicit.
\end{enumerate}

Clearly, the set of all $(A\rhd B)$-explicit sequents of $\pi_{\mathrm{L}}$ form a subtree of $\pi_{\mathrm{L}}$. 


Let 
	\[ \overline{\Theta}: = \left\{ E \rhd F\ \middle|\ 
		\begin{array}{l}
		\text{In the subtree consisting }(A \rhd B)\text{-explicit sequents in }\pi_\mathrm{L}, \\
		E \rhd F \text{ is the diagonal formula of an application of }(\rhd_\PP)\\
		\text{ in which }A \rhd B\text{ is principal}
		\end{array} \right\}.
	\] 
\begin{lem}\label{LemA}
For every $\Theta \subseteq \overline{\Theta}$ and $(A \rhd B)$-explicit sequent $S = (A \rhd B, \Gamma \Rightarrow \Delta)$ of $\pi_{\mathrm{L}}$, there exists an $(A, B)^\star$-proof of $(\Omega_1, \{ X_i \rhd Y_i \}, \Theta, \Gamma \Rightarrow \Delta)$.
\end{lem}
\begin{proof}
The lemma is proved by primary downward induction on the cardinality $|\Theta|$ of the set $\Theta \subseteq \overline{\Theta}$ and secondary top-down induction on the position of $S$ in the subtree of $(A\rhd B)$-explicit sequents of $\pi_\mathrm{L}$.

Assume $\Theta = \overline{\Theta}$. We show that the statement holds for every $(A \rhd B)$-explicit sequent $S$ of $\pi_L$ by top-down induction on the position of $S$.

\vspace{0.1in}
\textbf{Base case:} The sequent $S$ in question must be (1): an initial sequent $(A \rhd B \Rightarrow A \rhd B)$ or (2): an lower sequent of $\infer[(\WL)]{A \rhd B, \Gamma \Rightarrow \Delta}{\Gamma \Rightarrow \Delta}$ where $\Gamma$ does not contain $A \rhd B$. For Case (1), we get a cut-free proof of $(\Omega_1, \{ X_i \rhd Y_i \}, \overline{\Theta} \Rightarrow A \rhd B)$ by applying (wl) to the root of $\pi$ several times.
For Case (2), a cut-free proof of $(\Omega_1, \{ X_i \rhd Y_i \}, \overline{\Theta}, \Gamma \Rightarrow \Delta)$ is easily obtained by applying (wl) several times.

\vspace{0.1in}
\textbf{Inductive step:} (i) Suppose that $S$ is a lower sequent of an application of $(\rhd_\PP)$. We distinguish the following two cases:
	\begin{enumerate}
	\item If $A \rhd B$ is non-principal, then the application is of the form:
		\[
		\infer[(\rhd_\PP)]{\Omega', A \rhd B, \{ X_k' \rhd Y_k' \} \Rightarrow C' \rhd D'}
		{\Omega', A \rhd B, \{ X_k' \rhd Y_k' \}, C' \rhd D', C' \Rightarrow \{ X_k' \}
		&
		\langle Y_k' \Rightarrow D' \rangle
		}.
		\]
	By the secondary induction hypothesis, there exists an $(A,B)^\star$-proof of
	\[ \left(
		\begin{array}{cccc}
		\Omega_1, &  \{ X_i \rhd Y_i \}, &  \overline{\Theta}, & C' \rhd D', \\
		\Omega', & \{ X_k' \rhd Y_k' \}, & &  C'
		\end{array}
	\Rightarrow \{ X_k' \}\right).
	\]
	For this and $\langle Y_k' \Rightarrow D' \rangle$, we apply ($\rhd_\PP$). Notice that each sequent in
	$\langle Y_k' \Rightarrow D' \rangle$ has a cut-free proof because it is in $\pi_{\mathrm{L}}$. 
	Then, we obtain an $(A, B)^\star$-proof of
	$( \Omega_1, \Omega', \{ X_i \rhd Y_i \}, \{ X_k' \rhd Y_k' \}, \overline{\Theta} \Rightarrow C' \rhd D')$.
	Recall that $\overline{\Theta}$ contains only $\rhd$-formulas,
	and here the principal formulas in this application of $(\rhd_\PP)$ are $\{ X_k' \rhd Y_k' \}$. 

	\item If $A \rhd B$ is principal, then the application is of the form:
	\[ \infer[(\rhd_\PP)]{\Omega', \{ X_k' \rhd Y_k' \}, A \rhd B \Rightarrow E \rhd F}
	{\Omega', \{ X_k' \rhd Y_k' \}, A \rhd B, E \rhd F, E \Rightarrow \{ X_k' \}, A
	&
	\langle Y_k' \Rightarrow F \rangle
	&
	B \Rightarrow F
	}.
	\]
	Since $E \rhd F \in \overline{\Theta}$, we obtain a cut-free proof as follows:
	\[ \infer=[(\WL)\text{s}]{\Omega_1, \{ X_i \rhd Y_i \}, \overline{\Theta}, \Omega', \{X'_k \rhd Y'_k \} \Rightarrow E \rhd F}
	{E \rhd F \Rightarrow E \rhd F}.
	\]
	\end{enumerate}
(ii) The case that $S$ is a lower sequent of a logical rule is clear by the secondary induction hypothesis.
For instance, if $\dfrac{S_1 \quad S_2}{A \rhd B, \Gamma \Rightarrow \Delta}$ is some application of the logical rule $(\ast)$, then both $S_1$ and $S_2$ are also $(A \rhd B)$-explicit. 
Then, they are of the forms $(A \rhd B, \Gamma_1 \Rightarrow \Delta_1)$ and $(A \rhd B, \Gamma_2 \Rightarrow \Delta_2)$, respectively. 
By the secondary induction hypothesis, there exist $(A, B)^\star$-proofs of $(\Omega_1, \{X_i \rhd Y_i\}, \overline{\Theta}, \Gamma_1 \Rightarrow \Delta_1)$ and $(\Omega_1, \{X_i \rhd Y_i\}, \overline{\Theta}, \Gamma_2 \Rightarrow \Delta_2)$. 
Then, by the rule $(\ast)$, we obtain a required $(A, B)^\star$-proof of $(\Omega_1, \{X_i \rhd Y_i\}, \overline{\Theta}, \Gamma \Rightarrow \Delta)$. 

Now assume the lemma holds for subsets of $\overline{\Theta}$ with the cardinality $k+1$. Let $\Theta \subseteq \overline{\Theta}$ be such that $|\Theta| = k$. Again we show that the statement holds for every $(A\rhd B)$-explicit sequent $S$ of $\pi_\mathrm{L}$ by top-down induction on the position of $S$.

\vspace{0.1in}
\textbf{Base case:} Similar to the case $\Theta = \overline{\Theta}$.

\vspace{0.1in}
\textbf{Inductive step:} In the case that $S$ is a lower sequent of a logical rule, the statement can be proved similarly as before. It suffices to show the case that $S$ is a lower sequent of an application of $(\rhd_\PP)$. We distinguish the following two cases:
	\begin{enumerate}
	\item The formula $A \rhd B$ is non-principal: Similar to the case $\Theta = \overline{\Theta}$.
	\item If $A \rhd B$ is principal, then the application is of the form:
	\[ \infer[(\rhd_\PP)]{\Omega', \{ X_k' \rhd Y_k' \}, A \rhd B \Rightarrow E \rhd F}
	{\Omega', \{ X_k' \rhd Y_k' \}, A \rhd B, E \rhd F, E \Rightarrow \{ X_k' \}, A
	&
	\langle Y_k' \Rightarrow F \rangle
	&
	B \Rightarrow F
	}.
	\]
	If the diagonal formula $E \rhd F$ is contained in $\Theta$, then the statement is proved similarly as before by applying several ($\WL$)s.
	Suppose $E \rhd F \not\in \Theta$. Since $|\Theta \cup \{ E \rhd F \}| = k + 1$,
	there exists an $(A, B)^\star$-proof $\pi''$ of
	\[
	(\Omega_1, \{ X_i \rhd Y_i \}, \Theta, E \rhd F, A \Rightarrow \{ X_i \})
	\]
	by the primary induction hypothesis. (Here, we apply the hypothesis to the root of $\pi_\mathrm{L}$.)
	On the other hand, by the secondary induction hypothesis, we get an $(A, B)^\star$-proof of
	\[ ( \Omega_1,  \{ X_i \rhd Y_i \}, \Theta, \Omega', \{ X_k' \rhd Y_k' \}, E \rhd F, E
	\Rightarrow \{ X_k' \}, A ).
	\]
	Together with $\pi''$, we obtain an $(A, B)^\star$-proof of
	\[ (\Omega_1, \{ X_i \rhd Y_i \}, \Omega', \{ X_k' \rhd Y_k' \}, \Theta, E \rhd F, E \Rightarrow \{ X_i \}, \{ X_k' \} )\]
	by applying $(\mathrm{cut})$ with the cut-formula $A$. 
	Notice that each of $(Y_i \Rightarrow B)$, $(Y_k' \Rightarrow F)$, and $(B \Rightarrow F)$ has a cut-free proof. 
	Then, $(Y_i \Rightarrow F)$ has an $(A, B)^\star$-proof by applying the rule $(\mathrm{cut})$ with the cut-formula $B$. 
	Finally, we apply $(\rhd_\PP)$ as follows:
	\[ \infer[(\rhd_\PP)]{
		\begin{array}{cc}
		\Omega_1, & \{ X_i \rhd Y_i \}, \\
		\Omega', & \{ X_k' \rhd Y_k' \},
		\end{array}
	\Theta \Rightarrow E \rhd F}
	{
		\begin{array}{cc}
		\Omega_1, & \{ X_i \rhd Y_i \},  \\
		\Omega', & \{ X_k' \rhd Y_k' \},
		\end{array}
	\Theta, E \rhd F, E \Rightarrow 
		\begin{array}{c}
		\{ X_i \}, \\
		\{ X_k' \}
		\end{array}
	&
	\langle Y_i \Rightarrow F \rangle
	&
	\langle Y_k' \Rightarrow F \rangle
	}.
	\]
	Therefore, we have an $(A, B)^\star$-proof of
	$(\Omega_1, \{X_i \rhd Y_i\}, \Omega', \{X_k' \rhd Y_k'\}, \Theta \Rightarrow E \rhd F)$. \qedhere
	\end{enumerate}
\end{proof}
Applying Lemma~\ref{LemA} to $\Theta = \varnothing$ and the root of $\pi_L$, we get an $(A, B)^\star$-proof $\pi'$ of $(\Omega_1, \{ X_i \rhd Y_i \}, A \Rightarrow \{ X_i \})$.

\begin{proof}[Proof of Proposition~\ref{mainP}.]
Applying the rule $(\mathrm{cut})$ for $\pi$ and $\sigma_{\mathrm{L}}$ with the cut-formula $A \rhd B$, we obtain a proof of
\[
\left(
\begin{array}{ccc}
\Omega_1, & \{ X_i \rhd Y_i \}, & C \rhd D, \\
\Omega_2, & \{ Z_j \rhd W_j \}, & C
\end{array}
\Rightarrow \{ Z_j \}, A
\right).
\]
This application of $(\mathrm{cut})$ is eliminable by the induction hypothesis on the height. 
So we have a cut-free proof of the above sequent. 
From this and $\pi'$, we apply $(\mathrm{cut})$ with the cut-formula $A$. 
Then, we get an $(A, B)^\star$-proof of 
\begin{equation*}
\left(
\begin{array}{ccc}
\Omega_1, & \{ X_i \rhd Y_i \}, & C \rhd D, \\
\Omega_2, & \{ Z_j \rhd W_j \}, & C
\end{array}
\Rightarrow 
\begin{array}{c}
\{ X_i \}, \\
\{ Z_j \}
\end{array}
\right).
\end{equation*}
Apply $(\rhd_\PP)$ as follows:
\begin{equation*}
\infer[(\rhd_\PP)]{
\begin{array}{cc}
\Omega_1, & \{ X_i \rhd Y_i \}, \\
\Omega_2, & \{ Z_j \rhd W_j \}
\end{array}
\Rightarrow C \rhd D
}
{\left(
\begin{array}{ccc}
\Omega_1, & \{ X_i \rhd Y_i \}, & C \rhd D, \\
\Omega_2, & \{ Z_j \rhd W_j \}, & C
\end{array}
\Rightarrow 
\begin{array}{c}
\{ X_i \}, \\
\{ Z_j \}
\end{array}
\right)
&
\deduce{\langle Y_i \Rightarrow D \rangle}{\pi_i'}
&
\deduce{\langle W_j \Rightarrow D \rangle}{\sigma_j}
}.
\end{equation*}
Here $\pi_i'$ is obtained from $\pi_i$ and $\sigma_B$ by applying $(\mathrm{cut})$ with the cut-formula $B$. 
This is an $(A, B)^\star$-proof, and therefore the sequent $(\Omega_1, \Omega_2, \{ X_i \rhd Y_i \}, \{ Z_j \rhd W_j \} \Rightarrow C \rhd D)$ also has an $(A, B)^\star$-proof. 
The degree of every $(\mathrm{cut})$ in this proof is smaller than that of the original $(\mathrm{cut})$. 
Then, by the induction hypothesis, we obtain a cut-free proof of $(\Omega_1, \Omega_2, \{ X_i \rhd Y_i \}, \{ Z_j \rhd W_j \} \Rightarrow C \rhd D)$.
\end{proof}

\begin{thm}[The cut-elimination theorem for $\ILmPs$]\label{CE}
For any sequent $S$, if $S$ is provable in $\ILmPs$, then there exists a cut-free proof of $S$ in $\ILmPs$.
\end{thm}

\subsection{Craig interpolation property}

In this section, we show that Maehara's method of obtaining Craig interpolants of sequents having cut-free proofs can be applied to our systems $\ILms$ and $\ILmPs$. 
Therefore, we conclude that $\IL^-(\PP)$ enjoys CIP. 
Moreover, it follows that the cut-elimination theorem does not hold for our system $\ILms$. 

For a sequent $(\Gamma \Rightarrow \Delta)$, we say that a pair $[(\Gamma_1; \Delta_1), (\Gamma_2; \Delta_2)]$ of pairs of finite sets of formulas is a \textit{separation of $(\Gamma \Rightarrow \Delta)$} if $\Gamma_1 \cap \Gamma_2 = \Delta_1 \cap \Delta_2 = \varnothing$, $\Gamma = \Gamma_1 \cup \Gamma_2$, and $\Delta = \Delta_1 \cup \Delta_2$. 
For a finite set $\Gamma$ of formulas, let $\var(\Gamma) : = \bigcup\{\var(A) \mid A \in \Gamma\}$.

\begin{thm}\label{MM}
Let $\mathbf{L}$ be either $\ILms$ or $\ILmPs$. 
Suppose that $(\Gamma \Rightarrow \Delta)$ has a cut-free proof in $\mathbf{L}$. 
Then, for any separation $[(\Gamma_1; \Delta_1),(\Gamma_2; \Delta_2)]$ of $(\Gamma \Rightarrow \Delta)$, there exists a formula $C$ which satisfies the following conditions:
\begin{enumerate}
\item $\mathbf{L} \vdash (\Gamma_1 \Rightarrow \Delta_1, C)$ and $\mathbf{L} \vdash (\Gamma_2, C \Rightarrow \Delta_2)$;
\item $\var(C) \subseteq \var(\Gamma_1 \cup \Delta_1) \cap \var(\Gamma_2 \cup \Delta_2)$.
\end{enumerate}
\end{thm}

\begin{proof}
(i) Firstly, we prove the theorem for $\mathbf{L} = \ILmPs$ by top-down induction on the position of the sequent $(\Gamma \Rightarrow \Delta)$ in a cut-free proof $\pi$ of the sequent in $\ILmPs$. 
The cases for an initial sequent and an instance of logical rules are straightforward. 
We only give a proof of the case that the last application of $\pi$ is $(\rhd_\PP)$. 
Such an application of $(\rhd_\PP)$ is of the form:  
\begin{equation*}
\infer[(\rhd_\PP)]{
\begin{array}{cc}
\Omega_1, & \{ X_i \rhd Y_i \}, \\
\Omega_2, & \{ Z_j \rhd W_j \}
\end{array}
\Rightarrow A \rhd B}
{\begin{array}{cc}
\Omega_1, &  \{ X_i \rhd Y_i \}, \\
\Omega_2, & \{ Z_j \rhd W_j \}, 
\end{array}
A \rhd B, A \Rightarrow
\begin{array}{c}
\{ X_i \} , \\
\{ Z_j\}
\end{array}
&
\langle Y_i \Rightarrow B \rangle
&
\langle W_j \Rightarrow B \rangle
}.
\end{equation*}
Here we consider a separation $[(\Gamma_1; \Delta_1),(\Gamma_2; \Delta_2)]$ of the lower sequent $(\Omega_1, \Omega_2, \{X_i \rhd Y_i\}, \{Z_j \rhd W_j\} \Rightarrow A \rhd B)$ with $\Gamma_1 = \Omega_1 \cup \{ X_i \rhd Y_i \}$ and $\Gamma_2 = \Omega_2 \cup \{ Z_j \rhd W_j \}$. 
We distinguish the following two cases.

\begin{enumerate}
\item Suppose that $\Delta_1 = \{A \rhd B\}$ and $\Delta_2 = \varnothing$. 
Take the following separations of the upper sequents:
\begin{align*}
& [(\Omega_1, \{ X_i \rhd Y_i \}, A \rhd B, A; \{ X_i \}),(\Omega_2, \{ Z_j \rhd W_j \}; \{ Z_j \})], \\
& [(\varnothing ; B), (W_j; \varnothing)], \text{ for each }j.
\end{align*}
By the induction hypothesis, there exist $D$ and $E_j$ for each $j$ such that $\ILmPs$ proves
\begin{itemize}
	\item $(\Omega_1, \{ X_i \rhd Y_i \}, A \rhd B, A \Rightarrow \{ X_i \}, D)$, 
	\item $(D, \Omega_2, \{ Z_j \rhd W_j \} \Rightarrow \{ Z_j \})$,
	\item $(\Rightarrow B, E_j)$, and 
	\item $(E_j, W_j \Rightarrow)$.
\end{itemize}
Let $E'$ be the formula $\bigvee \neg E_j$. 
Then, we have $\ILmPs \vdash E' \Rightarrow B$ because $\ILmPs \vdash \neg E_j \Rightarrow B$ for each $j$. 
Notice that if $\{ Z_j \rhd W_j \}$ is an empty set, then $E'$ is $\bot$, and hence $\ILmPs \vdash E' \Rightarrow B$ also holds. 
Similarly, we also have $\ILmPs \vdash W_j \Rightarrow E'$ for each $j$. 
Consider the following derivations:
\begin{itemize}
	\item 
\[
\infer[(\NR)]{\Omega_1, \{ X_i \rhd Y_i \} \Rightarrow A \rhd B, \neg (D \rhd E')}
{\infer[(\rhd_\PP)]{\Omega_1, \{ X_i \rhd Y_i \}, D \rhd E' \Rightarrow A \rhd B}
{\infer[(\WL)]{\Omega_1, \{ X_i \rhd Y_i \}, D \rhd E', A \rhd B, A \Rightarrow \{ X_i \}, D}
{\Omega_1, \{ X_i \rhd Y_i \}, A \rhd B, A \Rightarrow \{ X_i \}, D}
&
\langle Y_i \Rightarrow B \rangle
&
E' \Rightarrow B
}}, 
\] 
	\item 
\[
\infer[(\NL)]{\Omega_2, \{ Z_j \rhd W_j \}, \neg (D \rhd E') \Rightarrow }
{\infer[(\rhd_\PP)]{\Omega_2, \{ Z_j \rhd W_j \} \Rightarrow D \rhd E'}
{\infer[(\WL)]{\Omega_2, \{ Z_j \rhd W_j \}, D \rhd E' , D \Rightarrow \{ Z_j \}}
{\Omega_2, \{ Z_j \rhd W_j \}, D \Rightarrow \{ Z_j \}}
&
\langle W_j \Rightarrow E' \rangle
}}.
\]
\end{itemize}
Then, $C : \equiv \neg (D \rhd E')$ fulfills Conditions 1 and 2 in the statement of the theorem.

\item Suppose that $\Delta_1 = \varnothing$ and $\Delta_2 = \{A \rhd B\}$. 
Take the following separations of the upper sequents:
\begin{align*}
& [(\Omega_1, \{ X_i \rhd Y_i \} ; \{ X_i \}), (\Omega_2, \{ Z_j \rhd W_j \}, A \rhd B, A ; \{ Z_j \} )], \\
& [(Y_i ; \varnothing ), (\varnothing ; B )], \text{ for each } i.
\end{align*}
By the induction hypothesis, there exist interpolants $D$ and $E_i$ for each $i$ such that $\ILmPs$ proves
\begin{itemize}
	\item $(\Omega_1, \{ X_i \rhd Y_i \} \Rightarrow \{ X_i \}, D)$, 
	\item $(D, \Omega_2, \{ Z_j \rhd W_j \}, A\rhd B, A \Rightarrow \{ Z_j \})$, 
	\item $(Y_i \Rightarrow E_i)$, and 
	\item $(E_i \Rightarrow B)$.
\end{itemize}
Let $E'$ be $\bigvee E_i$. 
Then, we have $\ILmPs \vdash (E' \Rightarrow B)$ and $\ILmPs \vdash (Y_i \Rightarrow E')$ for each $i$. 
Consider the following derivations:
\begin{itemize}
\item 
\[
\infer[(\rhd_\PP)]{\Omega_1, \{ X_i \rhd Y_i \} \Rightarrow (\neg D) \rhd E'}
	{\infer[(\WL)]{\Omega_1, \{ X_i \rhd Y_i \}, (\neg D) \rhd E', \neg D \Rightarrow \{ X_i \}}
		{\infer[(\NL)]{\Omega_1, \{ X_i \rhd Y_i \}, \neg D \Rightarrow \{ X_i \}}
			{\Omega_1, \{ X_i \rhd Y_i \} \Rightarrow \{ X_i \}, D}
		}
	&
	\langle Y_i \Rightarrow E' \rangle
	},
\]
\item 
\[
\infer[(\rhd_\PP)]{\Omega_2, \{ Z_j \rhd W_j \}, (\neg D) \rhd E' \Rightarrow A \rhd B}
	{\infer[(\WL)]{\Omega_2, \{ Z_j \rhd W_j \}, (\neg D) \rhd E', A \rhd B , A \Rightarrow \{ Z_j \}, \neg D}
		{\infer[(\NR)]{\Omega_2, \{ Z_j \rhd W_j \}, A \rhd B , A \Rightarrow \{ Z_j \}, \neg D}
			{D, \Omega_2, \{ Z_j \rhd W_j \}, A \rhd B, A \Rightarrow \{ Z_j \}}}
	&
	\langle W_j \Rightarrow B \rangle
	&
	E' \Rightarrow B
	}.
\]
\end{itemize}
Then, $C : \equiv (\neg D) \rhd E'$ fulfills Conditions 1 and 2 of the theorem.
\end{enumerate}
This finishes the proof of the theorem for $\mathbf{L} = \ILmPs$.

(ii) Secondly, we give a proof for $\mathbf{L} = \ILms$. 
A proof is similar as in the case of $\mathbf{L} = \ILmPs$, and so we only sketch a proof of the case of the rule $(\rhd)$. 
An instance of the rule is
\begin{equation*}
\infer[(\rhd)]{\{ X_i \rhd Y_i \}, \{ Z_j \rhd W_j \} \Rightarrow A \rhd B}
	{A \Rightarrow \{ X_i \} ,\{ Z_j\}
	&
	\langle Y_i \Rightarrow B \rangle
	&
	\langle W_j \Rightarrow B \rangle
	}, 
\end{equation*}
with $\Gamma_1 = \{ X_i \rhd Y_i \}$ and $\Gamma_2 = \{ Z_j \rhd W_j \}$. 
We distinguish the following two cases.
\begin{enumerate}
\item Suppose that $\Delta_1 = \{ A \rhd B \}$ and $\Delta_2 = \varnothing$. Take the following separations of the upper sequents:
\begin{align*}
& [(A ; \{X_i \} ) , (\varnothing ; \{ Z_j\})], \\
& [(\varnothing ; B) , (W_j ; \varnothing)] \text{ for each } j.
\end{align*}
Let $D$ and $E_j$ be corresponding interpolants obtained from the induction hypothesis, and put $E' :\equiv \bigvee \neg E_j$. 
The following derivations show that the formula $C \equiv \neg (D \rhd E')$ is indeed an interpolant of the separation in question:
\begin{itemize}
\item
\[ \infer[(\NR)]{\{ X_i \rhd Y_i \} \Rightarrow A \rhd B, \neg (D \rhd E')}
	{\infer[(\rhd)]{\{X_i \rhd Y_i \}, D \rhd E' \Rightarrow A \rhd B}
		{A \Rightarrow \{ X_i \}, D
		&
		\langle Y_i \Rightarrow B \rangle
		&
		E' \Rightarrow B
		}
	}, \]
\item 
\[ \infer[(\NL)]{\neg D \rhd E', \{ Z_j \rhd W_j \} \Rightarrow }
	{\infer[(\rhd)]{\{Z_j \rhd W_j \} \Rightarrow D \rhd E'}
		{D \Rightarrow \{ Z_j \}
		&
		\langle W_j \Rightarrow E' \rangle
		}
	}.\]
\end{itemize}
\item Suppose that $\Delta_1 = \varnothing$ and $\Delta_2 = \{ A \rhd B \}$. 
Take the following separations of the upper sequents:
\begin{align*}
& [(\varnothing ; \{ X_i \}), (A ; \{ Z_j \} )], \\
& [(Y_i ; \varnothing ), (\varnothing ; B )], \text{ for each } i.
\end{align*}
Let $D$ and $E_i$ be corresponding interpolants, and put $E' :\equiv \bigvee E_i$. 
The following derivations show that the formula $C \equiv (\neg D) \rhd E'$ is indeed an interpolant of the separation.
\begin{itemize}
\item \[ \infer[(\rhd)]{\{ X_i \rhd Y_i \} \Rightarrow (\neg D) \rhd E'}
	{\infer[(\NL)]{\neg D \Rightarrow \{ X_i \}}{ \Rightarrow \{X_i \}, D}
	&
	\langle Y_i \Rightarrow E' \rangle
	}, \]
\item \[ \infer[(\rhd)]{\{ Z_j \rhd W_j \}, (\neg D) \rhd E' \Rightarrow A \rhd B}
	{\infer[(\NR)]{A \Rightarrow \{Z_j \}, \neg D}{D, A \Rightarrow \{Z_j\}}
	&
	\langle W_j \Rightarrow B \rangle
	&
	E' \Rightarrow B
	}.\tag*{\mbox{\qedhere}}\]
\end{itemize}
\end{enumerate}
\end{proof}

\begin{thm}[The Craig interpolation theorem for $\IL^-(\PP)$]\label{CIP}
The logic $\IL^-(\PP)$ enjoys CIP.
\end{thm}

From Fact~\ref{CIPFPP} and this theorem, we obtain an alternative proof of Theorem~\ref{lFPP} stating that $\IL^-(\PP)$ has $\ell$FPP. 

According to \cite[Section 6]{IKO20}, $\IL^-$ does not have CIP. 
Therefore, we conclude the failure of the cut-elimination theorem for $\ILms$. 

\begin{cor}\label{FCE}
There exists a sequent $(\Gamma \Rightarrow \Delta)$ that is provable in $\ILms$ but cannot be proved without the rule $(\mathrm{cut})$ in $\ILms$. 
\end{cor}

\section{Relational semantic aspects}\label{Sec_SV}

In this section, we investigate relational semantics and the modal completeness for $\IL^-(\PP)$. 
This section consists of three subsections. 
In the first subsection, we introduce two kinds of relational semantics of $\IL^-(\PP)$, namely, $\IL^-\!$-frames and their simplifications. 
In particular, we show that $\IL^-(\PP)$ is valid in all such simplified frames. 
Thus, $\IL^-(\PP)$ is the natural basis for the investigation of a certain relational semantics. 
In the second subsection, we prove the modal completeness theorems of $\IL^-(\PP)$ with respect to these relational semantics. 
Finally, in the last subsection, we show that $\IL^-(\PP)$ is faithfully embeddable into several extensions of the fusion $\GLK$ of $\GL$ and $\K$. 



\subsection{Two relational semantics}

At first, we introduce \textit{$\IL^-\!$-frames} which were originally introduced by Visser~\cite{Vis88} as Veltman prestructures (See also~\cite[Definition 2.2]{KO21}).

\begin{defn}
We say that a triple $(W, R, \{S_{w}\}_{w \in W})$ is an \textit{$\IL^-\!$-frame} if the following conditions hold:
\begin{enumerate}
	\item $W$ is a non-empty set;
	\item $R$ is a transitive and conversely well-founded binary relation on $W$;
	\item For each $w \in W$, $S_{w} \subseteq R[w] \times W$ where $R[w] := \{x \in W \mid w {R} x\}$. 
\end{enumerate}
A quadruple $(W, R, \{S_{w}\}_{w \in W}, \Vdash)$ is said to be an \textit{$\IL^-\!$-model} if $(W, R, \{S_{w}\}_{w \in W})$ is an $\IL^- \!$-frame and $\Vdash$ is a binary relation between $W$ and the set of all $\mathcal{L}(\rhd)$-formulas satisfying the usual conditions for satisfaction relation and the following conditions:
\begin{itemize}
	\item $w \Vdash \Box A :\iff (\forall x \in W)(w {R} x \Rightarrow x \Vdash A)$; 
	\item $w \Vdash A \rhd B :\iff (\forall x \in W)\bigl(w {R} x \ \&\ x \Vdash A \Rightarrow (\exists y \in W)(x {S_{w}} y\ \&\ y \Vdash B)\bigr)$. 
\end{itemize}
A modal formula $A$ is said to be \textit{valid} in an $\IL^- \!$-frame $(W, R, \{S_{w}\}_{w \in W})$ if for all $\IL^- \!$-models ${(W, R, \{S_{w}\}_{w \in W}, \Vdash)}$ based on the frame and all $w \in W$, $w \Vdash A$. 
\end{defn}

It is easily shown that every theorem of $\IL^-$ is valid in all $\IL^-\!$-frames. 
Moreover, the modal completeness of $\IL^-$ with respect to $\IL^- \!$-frames has been proved. 

\begin{fact}[{\cite[Theorem 5.1]{KO21}}]
For any $\mathcal{L}(\rhd)$-formula $A$, the following are equivalent: 
\begin{enumerate}
	\item $\IL^- \vdash A$. 
	\item $A$ is valid in all $\IL^-\!$-frames. 
	\item $A$ is valid in all finite $\IL^- \!$-frames. 
\end{enumerate}
\end{fact}

The validity of an $\mathcal{L}(\rhd)$-formula in an $\IL^-\!$-frame is sometimes characterized by some condition on the binary relations on the frame (cf.~\cite{KO21,Vis88}). 
Regarding the principle $\PP$, it is easily shown that the following known condition on Veltman frames works also for $\IL^-\!$-frames (See Visser~\cite[p.~15]{Vis88}). 

\begin{prop}\label{FCP}
Let $\mathcal{F} = (W, R, \{S_{w}\}_{w \in W})$ be any $\IL^- \!$-frame. 
Then, the following are equivalent. 
\begin{enumerate}
	\item $\PP$ is valid in $\mathcal{F}$. 
	\item $(\forall w, x, y, z \in W) (w {R} x {R} y {S_{w}} z \Rightarrow y {S_{x}} z)$. 
\end{enumerate}
\end{prop}

The treatment of the family $\{S_w\}_{w \in W}$ in $\IL^-\!$-frames is somewhat complicated. 
The idea of simplifying them by converting them into a single relation $S$ was first done by Visser~\cite[Section 16]{Vis88} for $\IL^-\!$-frames where $\IL$ is valid. 
Our second semantics is obtained by implementing the idea on $\IL^-\!$-frames. 

\begin{defn}\label{Def:SILPF}
We say that a triple $(W, R, S)$ is a \textit{simplified $\IL^-(\PP)$-frame} if the following conditions hold:
\begin{enumerate}
	\item $W$ is a non-empty set;
	\item $R$ is a transitive and conversely well-founded binary relation on $W$;
	\item $S$ is a binary relation on $W$. 
\end{enumerate}
A \textit{simplified $\IL^-(\PP)$-model} is a quadruple $(W, R, S, \Vdash)$ where $(W, R, S)$ is a simplified $\IL^-(\PP)$-frame and $\Vdash$ is a satisfaction relation satisfying the following condition: 
\begin{itemize}
	\item[(a)] $w \Vdash A \rhd B :\iff (\forall x \in W)\bigl(w {R} x\ \&\ x \Vdash A \Rightarrow (\exists y \in W)(x {S} y\ \&\ y \Vdash B)\bigr)$.
\end{itemize}
\end{defn}

In the literature such as~\cite{Vis88}, a satisfaction relation $\Vdash$ on simplified frames is usually defined so that 
\begin{itemize}
	\item[(b)] $w \Vdash A \rhd B :\iff (\forall x \in W)\bigl(w {R} x\ \&\ x \Vdash A \Rightarrow (\exists y \in W)(w {R} y\ \&\ x {S} y\ \&\ y \Vdash B)\bigr)$. 
\end{itemize}
In contexts that adopt the definition (b), the logics under consideration contain the axiom scheme $\J{4}_+$, and models based on the definition (b) always validate $\J{4}_+$.
On the other hand, we adopt the definition (a) because we are dealing with logics that do not necessarily contain $\J{4}_+$ as an axiom scheme. 
However, in contrast to the case of $\IL^-\!$-frames, our simplified $\IL^-(\PP)$-frames always validate the persistence principle $\PP$. 
This is the reason why we adopted the terminology `simplified $\IL^-(\PP)$-frames' in Definition~\ref{Def:SILPF}. 
 
\begin{prop}\label{Psou}
The principle $\PP$ is valid in all simplified $\IL^-(\PP)$-frames. 
\end{prop}
\begin{proof}
Let $(W, R, S, \Vdash)$ be any $\IL^-(\PP)$-model and let $w \in W$. 
Suppose $w \Vdash A \rhd B$. 
We show $w \Vdash \Box(A \rhd B)$. 
Let $x, y \in W$ be such that $w {R} x$, $x {R} y$, and $y \Vdash A$. 
Since $w {R} y$ and $w \Vdash A \rhd B$, there exists a $z \in W$ such that $y {S} z$ and $z \Vdash B$. 
Therefore, we conclude $x \Vdash A \rhd B$ and we obtain $w \Vdash \Box(A \rhd B)$. 
\end{proof}

In the next subsection, we prove that the logic $\IL^-(\PP)$ is actually characterized by the class of all simplified $\IL^-(\PP)$-frames.

\subsection{Modal completeness}



In this subsection, we prove the modal completeness theorems for $\IL^-(\PP)$. 
Before proving our theorems, we prepare some definitions. 

For any set $\Phi$ of $\mathcal{L}(\rhd)$-formulas, let 
\[
	\Phi_{\rhd} := \{B \mid \text{there exists a} \, \, C \, \, \text{such that} \, \, B \rhd C \in \Phi \, \, \text{or} \, \, C \rhd B \in \Phi\}. 
\]
For any $\mathcal{L}(\rhd)$-formula $A$, let
\begin{eqnarray*}
{\sim} A :\equiv \left\{
\begin{array}{ll}
B & \text{if}\ A\ \text{is of the form}\ \neg B\ \text{for some}\ B, \\
\neg A & \text{otherwise.}
\end{array}
\right.
\end{eqnarray*}
A finite set $\Gamma$ of $\mathcal{L}(\rhd)$-formulas is said to be \textit{consistent} if $\IL^{-}(\PP) \nvdash \bigwedge \Gamma \to \bot$. 
Let $\Phi$ be any finite set of $\mathcal{L}(\rhd)$-formulas. 
A subset $\Gamma$ of $\Phi$ is said to be \textit{$\Phi$-maximally consistent} if $\Gamma$ is consistent and for any $A \in \Phi$, either $A \in \Gamma$ or ${\sim}A \in \Gamma$. 
Notice that if $X \subseteq \Phi$ is consistent, then there exists a $\Phi$-maximally consistent set $\Gamma$ including $X$. 
Moreover, if $\Gamma$ is a $\Phi$-maximally consistent set and $\IL^-(\PP) \vdash \bigwedge \Gamma \to A$ for $A \in \Phi$, then $A \in \Gamma$. 

\begin{defn}\label{Def:ade}
We say that a set $\Phi$ of $\mathcal{L}(\rhd)$-formulas is \textit{adequate} if it satisfies the following conditions:
\begin{enumerate}
	\item $\Phi$ is closed under taking subformulas and applying $\sim$;
	\item $\bot \in \Phi_{\rhd}$;
	\item If $B, C \in \Phi_{\rhd}$, then $B \rhd C \in \Phi$;
	\item If $B \rhd C \in \Phi$, then $\Box(B \rhd C) \in \Phi$;
	\item If $B \in \Phi_{\rhd}$, then $\Box{\sim}B \in \Phi$;
	\item If $B, C_{1}, \ldots, C_{n} \in \Phi_{\rhd}$, then $\Box(B \to \bigvee_{1 \leq i \leq n} C_{i}) \in \Phi$.
\end{enumerate}
\end{defn}

The following proposition is easily proved. 

\begin{prop}\label{finade}
Let $X$ be any finite set of $\mathcal{L}(\rhd)$formulas. 
Then, there exists a finite adequate set $\Phi$ including $X$. 
\end{prop}

Notice that the finiteness of $\Phi$ in Proposition~\ref{finade} is guaranteed by our setting that $\Box$ is in our language $\mathcal{L}(\rhd)$ as a primitive symbol, and $\Box A$ is not an abbreviation for $(\neg A) \rhd \bot$. 
We are ready to prove our first modal completeness theorem. 

\begin{thm}[The modal completeness theorem for $\IL^-(\PP)$ with respect to $\IL^-\!$-frames]
\label{MCILP}
For any $\mathcal{L}(\rhd)$-formula $A$, the following are equivalent: 
\begin{enumerate}
	\item $\IL^-(\PP) \vdash A$. 
	\item $A$ is valid in all $\IL^- \!$-frames in which $\PP$ is valid. 
	\item $A$ is valid in all finite $\IL^- \!$-frames in which $\PP$ is valid. 
\end{enumerate}
\end{thm}
\begin{proof}
The implications $(1 \Rightarrow 2)$ and $(2 \Rightarrow 3)$ are straightforward. 
We prove $(3 \Rightarrow 1)$. 
Suppose $\IL^-(\PP) \nvdash A$. 
We show that there exist a finite $\IL^- \!$-model $(W, R, \{S_{w}\}_{w \in W}, \Vdash)$ and a $w \in W$ such that $w \nVdash A$. 
There exists a finite adequate set $\Phi$ including $\{{\sim}A\}$ by Proposition~\ref{finade}. 
Let
\[
	K_{\Phi} := \{\Gamma \subseteq \Phi \mid \Gamma \ \text{is a} \ \Phi\text{-maximal consistent set}\}.
\]
Then, $K_{\Phi}$ is also a finite set. 

We define the binary relations $\prec$ and $\prec_C$ for each $C \in \Phi_\rhd$ on $W$ as follows: 
For $\Gamma, \Delta \in K_{\Phi}$, 
\begin{enumerate}
	\item $\Gamma \prec \Delta :\Longleftrightarrow$ $B, \Box B \in \Delta$ for any $\Box B \in \Gamma$, and there exists a $\Box C \in \Delta \setminus \Gamma$. 
	\item $\Gamma \prec_{C} \Delta :\Longleftrightarrow$ $\Gamma \prec \Delta$ and ${\sim}B \in \Delta$ for any $B \rhd C \in \Gamma$. 
\end{enumerate}

Notice that $\prec$ is transitive and irreflexive. 
Also, if $\Gamma \prec \Delta$, then $\Gamma \prec_{\bot} \Delta$ by Conditions 2 and 5 in Definition~\ref{Def:ade}. 

\begin{lem}\label{MClem1}
Let $\Gamma, \Delta, \Theta \in K_{\Phi}$ and $C \in \Phi_{\rhd}$. 
If $\Gamma \prec \Delta \prec_{C} \Theta$, then $\Gamma \prec_{C} \Theta$. 
\end{lem}

\begin{proof}
Suppose $\Gamma \prec \Delta \prec_{C} \Theta$. 
We have $\Gamma \prec \Theta$ by the transitivity of $\prec$. 
Let $B \rhd C \in \Gamma$. 
Since $\Box(B \rhd C) \in \Phi$ and $\IL^-(\PP) \vdash \bigwedge \Gamma \to \Box(B \rhd C)$, we have $\Box(B \rhd C) \in \Gamma$. 
Hence, $B \rhd C \in \Delta$ by $\Gamma \prec \Delta$. 
Therefore, ${\sim}B \in \Theta$ by $\Delta \prec_{C} \Theta$, and we conclude $\Gamma \prec_{C} \Theta$. 
\end{proof}

The following two facts were proved in~\cite{KO21}. 

\begin{fact}[{\cite[Lemma 4.6]{KO21}}]\label{MClem2}
Let $\Gamma \in K_{\Phi}$ and $C, D \in \Phi_{\rhd}$. 
If $C \rhd D \notin \Gamma$, then there exists a $\Delta \in K_{\Phi}$ such that $C \in \Delta$ and $\Gamma \prec_{D} \Delta$. 
\end{fact}

Notice that Condition 6 in Definition~\ref{Def:ade} is used in proving Fact~\ref{MClem2}. 

\begin{fact}[{\cite[Lemma 4.7]{KO21}}]\label{MClem3}
Let $\Gamma, \Delta \in K_{\Phi}$ and $C, D, E \in \Phi_{\rhd}$. 
If $C \rhd D \in \Gamma$, $\Gamma \prec_{E} \Delta$ and $C \in \Delta$, then there exists a $\Theta \in K_{\Phi}$ such that $D, {\sim}E \in \Theta$. 
\end{fact}

Let $(W, R, \{S_{w}\}_{w \in W}, \Vdash)$ be a quadruple defined as follows:
\begin{enumerate}
	\item $W :=\{(\Gamma, B) \mid \Gamma \in K_{\Phi}$ and $B \in \Phi_{\rhd} \}$;
	\item $(\Gamma, B) \mathrel{R} (\Delta, C) : \iff \Gamma \prec \Delta$;
	\item $(\Delta, C) \mathrel{S_{(\Gamma, B)}}  (\Theta, D) : \iff (\Gamma, B) \mathrel{R} (\Delta, C)$ and if $\Gamma \prec_{C} \Delta$, then ${\sim} C \in \Theta$; 
	\item $(\Gamma, B) \Vdash p : \iff p \in \Gamma$. 
\end{enumerate}

Let $\Gamma_{0}$ be a $\Phi$-maximally consistent set with ${\sim}A \in \Gamma_{0}$. 
Then, $W$ is a non-empty set because $(\Gamma_{0}, \bot) \in W$. 
Also, since $\Phi$ is finite, so is $W$. 
Obviously, $R$ is a transitive and irreflexive binary relation on $W$. 
Since $W$ is finite, $R$ is conversely well-founded. 
Trivially, $S_{(\Gamma, B)} \subseteq R[(\Gamma, B)] \times W$ for each $(\Gamma, B) \in W$. 
Therefore, $(W, R, \{S_{w}\}_{w \in W}, \Vdash)$ is a finite $\IL^-\!$-model. 

We show that $\PP$ is valid in $(W, R, \{S_{w}\}_{w \in W})$. 
Suppose $(\Gamma, B) \mathrel{R} (\Delta, C) \mathrel{R} (\Theta, D) \mathrel{S_{(\Gamma, B)}} (\Lambda, E)$. 
Suppose $\Delta \prec_{D} \Theta$. 
Since $\Gamma \prec \Delta \prec_{D} \Theta$, by Lemma~\ref{MClem1}, $\Gamma \prec_{D} \Theta$. 
By $(\Theta, D) \mathrel{S_{(\Gamma, B)}} (\Lambda, E)$, we have ${\sim} D \in \Lambda$. 
Therefore, we conclude $(\Theta, D) \mathrel{S_{(\Delta, C)}} (\Lambda, E)$. 
Then, $\PP$ is valid in $(W, R, \{S_{w}\}_{w \in W})$ by Proposition~\ref{FCP}. 

\begin{cl}
Let $A' \in \Phi$ and $(\Gamma, B) \in W$. 
Then, 
\[
(\Gamma, B) \Vdash A' \iff A' \in \Gamma. 
\]
\end{cl}
\begin{proof}
We prove the claim by induction on the construction of $A'$. 
We only prove the case that $A'$ is of the form $C \rhd D$. 

($\Rightarrow$): 
Suppose $C \rhd D \notin \Gamma$. 
By Fact~\ref{MClem2}, there exists a $\Delta \in K_{\Phi}$ such that $C \in \Delta$ and $\Gamma \prec_{D} \Delta$. 
Then, $(\Delta, D) \in W$ and $(\Gamma, B) \mathrel{R} (\Delta, D)$. 
Also, $(\Delta, D) \Vdash C$ by the induction hypothesis. 
Let $(\Theta, E) \in W$ be such that $(\Delta, D) \mathrel{S_{(\Gamma, B)}} (\Theta, E)$. 
Then, ${\sim}D \in \Theta$. 
By the induction hypothesis, we have $(\Theta, E) \nVdash D$. 
Therefore, we conclude $(\Gamma, B) \nVdash C \rhd D$. 

($\Leftarrow$): 
Suppose $C \rhd D \in \Gamma$ and let $(\Delta, E) \in W$ be such that $(\Gamma, B) \mathrel{R} (\Delta, E)$ and $(\Delta, E) \Vdash C$. 
By the induction hypothesis, $C \in \Delta$. 
We distinguish the following two cases. 
\begin{enumerate}
	\item Suppose $\Gamma \prec_{E} \Delta$. 
By Fact~\ref{MClem3}, there exists a $\Theta \in K_{\Phi}$ such that $D, {\sim}E \in \Theta$. 
Then, $(\Theta, E) \in W$ and $(\Delta, E) \mathrel{S_{(\Gamma, B)}} (\Theta, E)$. 
Also, $(\Theta, E) \Vdash D$ by the induction hypothesis. 
Therefore, we conclude $(\Gamma, B) \Vdash C \rhd D$. 

	\item Suppose $\Gamma \nprec_{E} \Delta$. 
Since $\Gamma \prec_{\bot} \Delta$, there exists a $\Theta \in K_{\Phi}$ such that $D \in \Theta$ by Fact~\ref{MClem3}. 
The rest of the proof is completely the same as in the first case. \qedhere
\end{enumerate}
\end{proof}
Since $A \notin \Gamma_0$, we conclude $(\Gamma_{0}, \bot) \nVdash A$ by the claim. 
This finishes our proof of Theorem~\ref{MCILP}. 
\end{proof}

Next, we prove the modal completeness for $\IL^-(\PP)$ with respect to simplified $\IL^-(\PP)$-frames.
Moreover, we consider the following condition $(\dagger)$ on simplified $\IL^-(\PP)$-frames $(W, R, S)$: 
\[
	\text{There exist no}\ x, y, z \in W\ \text{such that}\ x S y S z. \hspace{1in} (\dagger)
\] 

\begin{thm}\label{SVC}
For any $\mathcal{L}(\rhd)$-formula $A$, the following are equivalent: 
\begin{enumerate}
	\item $\IL^-(\PP) \vdash A$. 
	\item $A$ is valid in all simplified $\IL^-(\PP)$-frames. 
	\item $A$ is valid in all finite simplified $\IL^-(\PP)$-frames satisfying the condition \textup{($\dagger$)}. 
\end{enumerate}
\end{thm}

\begin{proof}
$(1 \Rightarrow 2)$: This is immediate from Proposition~\ref{Psou}. 

$(2 \Rightarrow 3)$: Trivial. 

$(3 \Rightarrow 1)$: 
Suppose $\IL^-(\PP) \nvdash A$. 
Then, by Theorem~\ref{MCILP}, there exist a finite $\IL^- \!$-model $(W, R, \{S_{w}\}_{w \in W}, \Vdash)$ and $w_0 \in W$ such that $\PP$ is valid in $(W, R, \{S_{w}\})$ and $w_0 \nVdash A$. 
We would like to find some finite simplified $\IL^-(\PP)$-model $(W', R', S', \Vdash')$ and $w' \in W'$ such that $w' \nVdash' A$. 
Let $(W', R', S', \Vdash')$ be a quadruple satisfying the following conditions:
\begin{enumerate}
	\item $W' :=\{\seq{w_1, \ldots, w_n} \mid n \geq 1, (\forall i \leq n)(w_{i} \in W), \, \, \text{and} \, \, (\forall i < n)(w_{i} {R} w_{i+1}) \}$;
	\item $\seq{x_1, \ldots, x_n} \mathrel{R'} \seq{y_1, \ldots, y_m} :\iff n < m$ and $(\forall i \leq n)(x_i = y_i)$;
	\item $\seq{x_1, \ldots, x_n} \mathrel{S'} \seq{y_1, \ldots, y_m} :\iff n > 1$, $m=1$, and $x_{n} {S_{x_{n-1}}} y_{m}$; 
	\item $\seq{w_1, \ldots, w_n} \Vdash' p :\iff w_n \Vdash p$. 
\end{enumerate}

$W'$ is a non-empty set because $\seq{w_0} \in W'$. 
Furthermore, since $W$ is finite and $R$ is a conversely well-founded binary relation, $W'$ is also finite. 
Therefore, $(W', R', S', \Vdash')$ is a finite simplified $\IL^-(\PP)$-model. 
For any elements $\seq{x_1, \ldots, x_n}$, $\seq{y_1, \ldots, y_m}$, $\seq{z_1, \ldots, z_k}$ of $W'$, if $\seq{x_1, \ldots, x_n} \mathrel{S'} \seq{y_1, \ldots, y_m}$, then $m = 1$, and hence $\seq{y_1, \ldots, y_m}\mathrel{S'} \seq{z_1, \ldots, z_k}$ does not hold. 
This means that the frame $(W', R', S')$ satisfies the condition $(\dagger)$. 

\begin{cl}
Let $A'$ be any $\mathcal{L}(\rhd)$-formula and let $\seq{w_1, \ldots, w_n} \in W'$. 
Then, 
\[
\seq{w_1, \ldots, w_n} \Vdash' A' \iff w_n \Vdash A'. 
\]
\end{cl}
\begin{proof}
We prove the claim by induction on the construction of $A'$. 
We only give a proof of the case that $A'$ is of the form $B \rhd C$. 

($\Rightarrow$): 
Suppose $\seq{w_1, \ldots, w_n} \Vdash' B \rhd C$. 
Let $x \in W$ be such that $w_{n} {R} x$ and $x \Vdash B$. 
Then, we have $\seq{w_1, \ldots, w_n, x} \in W'$ and $\seq{w_1, \ldots, w_n} \mathrel{R'} \seq{w_1, \ldots, w_n, x}$. 
By the induction hypothesis, $\seq{w_1, \ldots, w_n, x} \Vdash B$. 
By our supposition, there exists a $\seq{y} \in W'$ such that $\seq{w_1, \ldots, w_n, x} \mathrel{S'} \seq{y} \Vdash' C$. 
By the definition of $S'$ and the induction hypothesis, we have $x {S_{w_{n}}} y$ and $y \Vdash C$. 
Hence, $w_n \Vdash B \rhd C$. 

($\Leftarrow$): Suppose $w_n \Vdash B \rhd C$. 
Let $\seq{x_1,\ldots, x_{m}} \in W'$ be such that $\seq{w_1, \ldots, w_n} \mathrel{R'} \seq{x_1,\ldots, x_{m}}$ and $\seq{x_1,\ldots, x_{m}} \Vdash' B$.
By the definition of $R'$ and the induction hypothesis, we have $n < m$, $w_{n} {R} x_{m}$ and $x_m \Vdash B$. 
By our supposition, there exists a $y \in W$ such that $x_{m} {S_{w_{n}}} y$ and $y \Vdash C$. 
Since $n < m$, either $n = m-1$ or $n < m-1$. 
Therefore, we have either $w_{n} = x_{m-1}$ or $w_{n} {R} x_{m-1}$ by the definition of $R'$. 
If $w_{n} = x_{m-1}$, then $x_{m} {S_{x_{m-1}}} y$ is obvious. 
If $w_{n} {R} x_{m-1}$, then $x_{m} {S_{x_{m-1}}} y$ holds by Proposition~\ref{FCP} because $\PP$ is valid in $(W, R, \{S_{w}\}_{w \in W})$. 
In either case, we obtain $x_{m} {S_{x_{m-1}}} y$. 
Then, $\seq{x_1,\ldots, x_{m}} \mathrel{S'} \seq{y}$ because $m > n \geq 1$. 
Also, we have $\seq{y} \Vdash' C$ by the induction hypothesis. 
Therefore, we conclude $\seq{w_1, \ldots, w_n} \Vdash' B \rhd C$. 
\end{proof}

Since $w_{0} \nVdash A$, we obtain $\seq{w_{0}} \nVdash' A$ by the claim.  
\end{proof}

\subsection{An embedding of the weak interpretability logic with persistence into bimodal logics}

In this subsection, as an application of Theorem~\ref{SVC}, we show that $\IL^-(\PP)$ is faithfully embedded into several bimodal logics. 
The language $\mathcal{L}_2$ of bimodal propositional logic is the language of propositional logic equipped with two unary modal operators $[0]$ and $[1]$. 
The logic $\GLK$ in the language $\mathcal{L}_2$ has the following axioms and rules: 
\begin{itemize}
	\item All tautologies in the language $\mathcal{L}_2$; 
	\item $[0] (A \to B) \to ([0] A \to [0] B)$; 
	\item $[0] ([0] A \to A) \to [0] A$; 
	\item $[1] (A \to B) \to ([1] A \to [1] B)$; 
	\item $\dfrac{A \to B \quad A}{B}$; 
	\item $\dfrac{A}{[k] A}$ for $k \in \{0, 1\}$. 
\end{itemize}
The logic $\GLK$ is called the \textit{fusion} of $\GL$ and $\K$. 
Let $(\GLK) \oplus [1][1]\bot$ be the logic obtained from $\GLK$ by adding $[1][1]\bot$ as an axiom. 

We say that a triple $(W, R_0, R_1)$ is a \textit{$\GLK$-frame} if $W$ is a nonempty set, $R_0$ and $R_1$ are binary relations on $W$, and $R_0$ is transitive and conversely well-founded. 
Here for each $k \in \{0, 1\}$, the binary relation $R_k$ corresponds to the modal operator $[k]$. 
From a general result about fusions of modal logics, it is known that $\GLK$ is characterized by the class of all finite $\GLK$-frames (cf.~Kurucz~\cite[Theorem 3]{Kuru07}). 

We introduce a translation $\chi$ from $\mathcal{L}(\rhd)$-formulas into $\mathcal{L}_2$-formulas defined as follows: 
\begin{enumerate}
	\item $\chi(\bot)$ is $\bot$; 
	\item $\chi(p)$ is $p$ for each propositional variable $p$; 
	\item $\chi(\neg A)$ is $\neg \chi(A)$; 
	\item $\chi(A \circ B)$ is $\chi(A) \circ \chi(B)$ for $\circ \in \{\land, \lor, \to\}$; 
	\item $\chi(\Box A)$ is $[0] \chi(A)$; 
	\item $\chi(A \rhd B)$ is $[0]\bigl(\chi(A) \to \langle 1 \rangle \chi(B)\bigr)$. 
\end{enumerate}
Here $\langle 1 \rangle$ is the abbreviation for $\neg [1] \neg$. 
Then, we prove the following embedding result. 

\begin{prop}\label{Embed}
For any modal formula $A$, the following are equivalent: 
\begin{enumerate}
	\item $\IL^-(\PP) \vdash A$. 
	\item $\GLK \vdash \chi(A)$. 
	\item $(\GLK) \oplus [1][1]\bot \vdash \chi(A)$. 
\end{enumerate}
\end{prop}
\begin{proof}
$(1 \Rightarrow 2)$: 
We prove the implication by induction on the length of proofs in $\IL^-(\PP)$. 
\begin{itemize}
	\item If $A$ is one of $\G{1}$, $\G{2}$, and $\G{3}$, then $\chi(A)$ is also an axiom of $\GLK$. 
	\item ($\J{3}$): Since 
\begin{align*}
	\bigl(\chi(A) \to \langle 1 \rangle \chi(C)\bigr) \land \bigl(\chi(B) \to \langle 1 \rangle \chi(C)\bigr) \to \bigl(\chi(A \lor B) \to \langle 1 \rangle \chi(C)\bigr)
\end{align*}
is a tautology, we have 
\[
	\GLK \vdash \chi(A \rhd C) \land \chi(B \rhd C) \to \chi((A \lor B) \rhd C). 
\]
	\item ($\J{6}$): Notice that $\chi((\neg A) \rhd \bot)$ is $[0] (\neg \chi(A) \to \langle 1 \rangle \bot)$. 
	Since $\GLK \vdash \neg \langle 1 \rangle \bot$, $\GLK \vdash \chi((\neg A) \rhd \bot) \leftrightarrow \chi(\Box A)$. 
	\item ($\PP$): Since $\chi(A \rhd B)$ is of the form $[0] C$ for some $C$, $\GLK \vdash \chi(A \rhd B) \to \chi(\Box (A \rhd B))$. 

	\item If $A$ is derived by using Modus Ponens or Necessitation, then $\GLK \vdash \chi(A)$ is obvious by using the induction hypothesis. 

	\item If $A$ is derived by using the rule $\R{1}$ from $B \to C$, then $A$ is of the form $D \rhd B \to D \rhd C$ for some $D$. 
By the induction hypothesis, $\GLK \vdash \chi(B) \to \chi(C)$. 
Then, $\GLK \vdash \langle 1 \rangle \chi(B) \to \langle 1 \rangle \chi(C)$, and hence
\[
	\GLK \vdash \bigl(\chi(D) \to \langle 1 \rangle \chi(B)\bigr) \to \bigl(\chi(D) \to \langle 1 \rangle \chi(C)\bigr). 
\]
Thus, $\GLK \vdash \chi(D \rhd B) \to \chi(D \rhd C)$. 

	\item If $A$ is derived from $B \to C$ by using the rule $\R{2}$, then $A$ is of the form $C \rhd D \to B \rhd D$. 
By the induction hypothesis, $\GLK \vdash \chi(B) \to \chi(C)$, and then
\[
	\GLK \vdash \bigl(\chi(C) \to \langle 1 \rangle \chi(D)\bigr) \to \bigl(\chi(B) \to \langle 1 \rangle \chi(D)\bigr). 
\]
We obtain $\GLK \vdash \chi(C \rhd D) \to \chi(B \rhd D)$. 
\end{itemize}

$(2 \Rightarrow 3)$: Obvious. 

$(3 \Rightarrow 1)$: Suppose $\IL^-(\PP) \nvdash A$. 
By Theorem~\ref{SVC}, there exists a finite simplified $\IL^-(\PP)$-model $(W, R, S, \Vdash)$ satisfying the condition ($\dagger$) and an element $w \in W$ such that $w \nVdash A$. 
Then, $(W, R, S)$ is also a $\GLK$-frame. 
Let $\Vdash^*$ be a satisfaction relation on the $\GLK$-frame $(W, R, S)$ satisfying $x \Vdash^* p \iff x \Vdash p$ for any $x \in W$ and propositional variable $p$. 
Then, it is shown by induction on the construction of $B$ that for any $\mathcal{L}(\rhd)$-formula $B$ and $x \in W$, $x \Vdash^* \chi(B)$ if and only if $x \Vdash B$. 
Hence, we obtain $w \nVdash^* \chi(A)$. 
Also, by the condition $(\dagger)$, $[1][1]\bot$ is valid in the $\GLK$-frame $(W, R, S)$. 
Therefore, every theorem of $(\GLK) \oplus [1][1]\bot$ is valid in the frame. 
Hence, we conclude that $(\GLK) \oplus [1][1] \bot \nvdash \chi(A)$. 
\end{proof}

From this embedding result, $\IL^-(\PP)$ is faithfully embeddable into bimodal logics $L$ between the logics $\GLK$ and $(\GLK) \oplus [1][1]\bot$. 
Prominent examples of such logics $L$ are $\GL \otimes \mathbf{K4} = (\GLK) \oplus ([1]A \to [1][1]A)$ and $\FGL = (\GLK) \oplus ([1]([1]A \to A) \to [1]A)$. 

Finally, we prove the failure of FPP for the logic $(\GLK) \oplus [1][1] \bot$. 

\begin{prop}\label{FailureofFPP}
For any $\mathcal{L}_2$-formula $A$, 
\[
	(\GLK) \oplus [1][1] \bot \nvdash A \leftrightarrow [0] \neg [1] A.
\]
Consequently, the logic $(\GLK) \oplus [1][1] \bot$ does not enjoy FPP. 
\end{prop}
\begin{proof}
Let $(W, R_0, R_1)$ be a triple with $W : = \{x, y\}$, $R_0 := \{(x, y)\}$, and $R_1 := \{(y, x)\}$. 
It is easily shown that $(W, R_0, R_1)$ is a $\GLK$-frame in which $[1][1] \bot$ is valid. 
Let $A$ be any $\mathcal{L}_2$-formula and $\Vdash$ be any satisfaction relation on the frame. 
Since $x \Vdash A$ if and only if $y \Vdash [1]A$, and $y \Vdash [1]A$ if and only if $x \Vdash \neg [0] \neg [1] A$, we obtain that $x \Vdash A \leftrightarrow \neg [0] \neg [1] A$. 
Therefore, $x \nVdash A \leftrightarrow [0] \neg [1] A$. 
We conclude $(\GLK) \oplus [1][1]\bot \nvdash A \leftrightarrow [0] \neg [1] A$. 
\end{proof}

As a benefit of our embedding result, we get the following corollary in contrast to Theorem~\ref{lFPP}.

\begin{cor}\label{FailureofFPP2}
For any $\mathcal{L}(\rhd)$-formula $A$,  
\[
\IL^-(\PP) \nvdash A \leftrightarrow \top \rhd \lnot A. 
\] 
Hence, $\IL^-(\PP)$ does not enjoy FPP. 
\end{cor}
\begin{proof}
Suppose, towards a contradiction, that $\IL^-(\PP) \vdash A \leftrightarrow \top \rhd \lnot A$ for some $\mathcal{L}(\rhd)$-formula $A$. 
Then, by Proposition~\ref{Embed}, $(\GLK) \oplus [1][1]\bot \vdash \chi(A) \leftrightarrow \chi(\top \rhd \neg A)$. 
It follows that $(\GLK) \oplus [1][1]\bot$ proves $\chi(A) \leftrightarrow [0] \neg [1] \chi(A)$. 
This contradicts Proposition~\ref{FailureofFPP}. 
\end{proof}

We say that the \textit{uniqueness of fixed points (UFP)} holds for a logic $L$ if for any $\mathcal{L}(\rhd)$-formula $A(p)$ such that $p$ is modalized in $A(p)$ and $A(p)$ does not contain $q$,
\[
	L \vdash \boxdot(p \leftrightarrow A(p) ) \land \boxdot (q \leftrightarrow A(q) ) \to ( p \leftrightarrow q),
\]
where $\boxdot F$ is an abbreviation for $\Box F \land F$. 
It was essentially proved in de Jongh and Visser~\cite[Theorem 2.1]{DeJVis91} that UFP holds for $\IL^-(\J{4}_+)$.  
As shown in~\cite[Lemma 3.10]{IKO20}, for any extension $L$ of $\IL^-$, if $L$ has CIP and UFP holds for $L$, then $L$ has FPP. 
This explains why Fact~\ref{CIPFPP}.2 holds. 
From Theorem~\ref{CIP} and Corollary~\ref{FailureofFPP2}, we obtain that UFP does not hold for $\IL^-(\PP)$.

\section{Arithmetical semantic aspects}\label{Sec_AC}

In this section, we investigate arithmetical semantics for $\IL^-(\PP)$. 
Firstly, we introduce appropriate arithmetical interpretations for $\IL^-(\PP)$ inspired from our embedding $\chi$. 
Secondly, we strengthen Theorem~\ref{SVC}. 
Then, by using such a strengthening, we prove the arithmetical completeness theorem for $\IL^-(\PP)$. 

Throughout this section, we assume that $T$ always denotes a consistent recursively enumerable extension of $\PA$ in the language $\mathcal{L}_A$ of arithmetic. 
We say that a formula $\tau(u)$ is a \textit{numeration} of $T$ if for any natural number $n$, $n$ is the G\"odel number of an element of $T$ if and only if $\PA \vdash \tau(\overline{n})$. 
For each numeration $\tau(u)$ of $T$, we can construct a formula $\Prf_{\tau}(x, y)$ saying that $y$ is the G\"odel number of a proof of a sentence whose G\"odel number is $x$ from the theory axiomatized by the set of sentences defined by $\tau$. 
Let $\PR_\tau(x)$ be the formula $\exists y \Prf_\tau(x, y)$. 
Then, the following fact holds: 

\begin{fact}[Feferman~\cite{Fef60}]\label{DC}
Let $\tau(u)$ be any numeration of $T$ and let $\varphi$ and $\psi$ be any $\mathcal{L}_A$-formulas. 
\begin{enumerate}
	\item If $T \vdash \varphi$, then $\PA \vdash \PR_\tau(\gn{\varphi})$. 
	\item $\PA \vdash \PR_\tau(\gn{\varphi \to \psi}) \to (\PR_\tau(\gn{\varphi}) \to \PR_\tau(\gn{\psi}))$. 
	\item If $\varphi$ is a $\Sigma_1$ sentence, then $\PA \vdash \varphi \to \PR_\tau(\gn{\varphi})$. 
\end{enumerate}
\end{fact}

If $\tau(u)$ is a $\Sigma_1$ numeration of $T$, then it is known that $\PR_\tau(x)$ is also $\Sigma_1$. 
Hence, $\PA \vdash \PR_\tau(\gn{\varphi}) \to \PR_\tau(\gn{\PR_\tau(\gn{\varphi})})$ holds for any $\mathcal{L}_A$-formula $\varphi$. 
On the other hand, there exist $\Sigma_2$ numerations $\tau(u)$ of $T$ such that $T \nvdash \PR_\tau(\gn{\varphi}) \to \PR_\tau(\gn{\PR_\tau(\gn{\varphi})})$ (see~\cite{Kur18_1,Kur18_2,Vis21}). 
We say that a numeration $\tau(u)$ of $T$ satisfies the \textit{L\"ob condition} if $\PA \vdash \PR_\tau(\gn{\varphi}) \to \PR_\tau(\gn{\PR_\tau(\gn{\varphi})})$ for any $\mathcal{L}_A$-formula $\varphi$. 

A mapping $f$ from the set of all propositional variables to a set of $\mathcal{L}_A$-sentences is called an \textit{arithmetical interpretation}. 
For any pair $(\tau_0, \tau_1)$ of numerations of $T$, every arithmetical interpretation $f$ is extended to a mapping $f'$ whose domain is the set of all $\mathcal{L}_2$-formulas by the following clauses: 
\begin{enumerate}
	\item $f'(\bot)$ is $0 = 1$; 
	\item $f'(\neg A)$ is $\neg f'(A)$; 
	\item $f'(A \circ B)$ is $f'(A) \circ f'(B)$ for $\circ \in \{\land, \lor, \leftrightarrow\}$; 
	\item $f'([k] A)$ is $\PR_{\tau_k}(\gn{f'(A)})$ for $k \in \{0, 1\}$. 
\end{enumerate}
In particular, it was proved that $\mathbf{GLK} = (\GLK) \oplus ([0]A \to [1]A)$ is arithmetically sound and complete with respect to the classes of all pairs $(\tau_0, \tau_1)$ of numerations such that $\tau_0(u)$ is $\Sigma_1$ and $\PA \vdash \forall x(\PR_{\tau_0}(x) \to \PR_{\tau_1}(x))$ (\cite[Corollary 4.11]{Kur18_1}). 
Also, Beklemishev~\cite[Theorem 1]{Bek92} proved that $\mathbf{CS}_1 = (\FGL) \oplus ([0]A \to [1][0]A) \oplus ([1]A \to [0][1]A)$ and its extensions are arithmetically sound and complete with respect to some appropriate classes of pairs of $\Sigma_1$ numerations. 

Here we extend arithmetical interpretations to mappings of $\mathcal{L}(\rhd)$-formulas inspired by our embedding result of $\IL^-(\PP)$ into bimodal logics. 
Since the translation $\chi$ introduced in the last section is such an embedding, the mapping $f' \circ \chi$ seems to be appropriate. 
In fact, by Proposition~\ref{Embed}, for any theorem $A$ of $\IL^-(\PP)$, we obtain $\GLK \vdash \chi(A)$, and if $\tau_0$ satisfies the L\"ob condition, then it is easily shown that $\PA \vdash (f' \circ \chi)(A)$. 

From this observation, we directly extend every arithmetical interpretation $f$ to a mapping $f_{\tau_0, \tau_1}$ from the set of all $\mathcal{L}(\rhd)$-formulas to a set of $\mathcal{L}_A$-sentences as follows: 
\begin{enumerate}
	\item $\FTT{\bot}$ is $0 = 1$; 
	\item $\FTT{\neg A}$ is $\neg \FTT{A}$; 
	\item $\FTT{A \circ B}$ is $\FTT{A} \circ \FTT{B}$ for $\circ \in \{\land, \lor, \leftrightarrow\}$; 
	\item $\FTT{\Box A}$ is $\PR_{\tau_0}(\gn{\FTT{A}})$; 
	\item $\FTT{A \rhd B}$ is $\PR_{\tau_0}(\gn{\FTT{A} \to \Con_{\tau_1 + \FTT{B}}})$. 
\end{enumerate}
Here for each $\mathcal{L}_A$-sentence $\psi$, $(\tau_1 + \psi)(u)$ is the numeration $\tau_1(u) \lor u = \gn{\psi}$ of the theory $T + \psi$. 

The following arithmetical soundness of $\IL^-(\PP)$ follows from the above argument. 

\begin{prop}[Arithmetical soundness for $\IL^-(\PP)$]\label{AS}
Let $A$ be any $\mathcal{L}(\rhd)$-formula $A$ and $(\tau_0, \tau_1)$ be any pair of numerations of $T$. 
If $\IL^-(\PP) \vdash A$ and $\tau_0$ satisfies the L\"ob condition, then for any arithmetical interpretation $f$, $\PA \vdash \FTT{A}$. 
\end{prop}

The remainder of this section is devoted to proving the converse of Proposition~\ref{AS}. 
Furthermore, we prove the following uniform version of the arithmetical completeness theorem. 

\begin{thm}[Uniform arithmetical completeness theorem for $\IL^-(\PP)$]\label{AC}
There exist a pair $(\tau_0, \tau_1)$ of $\Sigma_2$ numerations of $T$ such that both $\tau_0$ and $\tau_1$ satisfy the L\"ob condition and an arithmetical interpretation $f$ such that for any $\mathcal{L}(\rhd)$-formula $A$, the following are equivalent: 
\begin{enumerate}
	\item $\IL^-(\PP) \vdash A$. 
	\item $\PA \vdash \FTT{A}$. 
	\item $T \vdash \FTT{A}$. 
\end{enumerate}
\end{thm}

From Proposition~\ref{AS} and Theorem~\ref{AC}, we obtain the following corollary. 

\begin{cor}
For any $\mathcal{L}(\rhd)$-formula $A$, the following are equivalent: 
\begin{enumerate}
	\item $\IL^-(\PP) \vdash A$. 
	\item For any pair $(\tau_0, \tau_1)$ of numerations of $T$ in which $\tau_0$ satisfies the L\"ob condition and any arithmetical interpretation $f$, $\PA \vdash \FTT{\chi(A)}$. 
	\item For any pair $(\tau_0, \tau_1)$ of $\Sigma_2$ numerations of $T$ such that both $\tau_0$ and $\tau_1$ satisfy the L\"ob condition and any arithmetical interpretation $f$, $T \vdash \FTT{\chi(A)}$. 
\end{enumerate}
\end{cor}

Before proving Theorem~\ref{AC}, we strengthen Theorem~\ref{SVC}. 

\begin{defn}
For each $\mathcal{L}(\rhd)$-formula $A$, we define the natural number $d(A)$ recursively as follows: 
\begin{enumerate}
	\item $d(p) = d(\bot) = 0$; 
	\item $d(A \circ B) = \max\{d(A), d(B)\}$ for $\circ \in \{\land, \lor, \to\}$; 
	\item $d(\neg A) = d(\Box A) = d(A)$; 
	\item $d(A \rhd B) = \max\{d(A), d(B)+1\}$. 
\end{enumerate}
\end{defn}

\begin{thm}\label{SVC2}
The logic $\IL^-(\PP)$ is sound and complete with respect to the class of all finite simplified $\IL^-(\PP)$-frames $(W, R, S)$ such that $S$ is transitive and $R \cup S$ is conversely well-founded. 
\end{thm}
\begin{proof}
Suppose $\IL^-(\PP) \nvdash A$, then by Theorem~\ref{SVC}, there exists a finite simplified $\IL^-(\PP)$-model $\mathcal{M} = (W, R, S, \Vdash)$ satisfying the condition $(\dagger)$ and an element $w \in W$ such that $w \nVdash A$. 
We define a new simplified $\IL^-(\PP)$-model $\mathcal{M}^* = (W^*, R^*, S^*, \Vdash^*)$ as follows: 
\begin{itemize}
	\item $W^* : = \{(x, n) \mid x \in W$ and $0 \leq n \leq d(A)\}$;	
	\item $(x, n) R^* (y, m) : \iff x R y$ and $n = m$; 
	\item $(x, n) S^* (y, m) : \iff x S y$ and $n = m+1$; 
	\item $(x, n) \Vdash^* p : \iff x \Vdash p$. 
\end{itemize}

The frame $(W^*, R^*, S^*)$ consists of the $d(A)+1$ copies of the original frame $(W, R, S)$ with different levels as indicated in the second components of elements of $W$.  
The $R^*$-transition does not change the level, and the $S^*$-transition reduces the level by one. 

By the condition $(\dagger)$ for $(W, R, S)$ and the definition of our $S^*$, the frame $(W^*, R^*, S^*)$ also satisfies the condition $(\dagger)$. 
In particular, we have that $S^*$ is transitive. 

The value of the second component of each element of $W^*$ is not changed by the $R^*$-transition, but it is decreased by the $S^*$-transition. 
So every $R^* \cup S^*$-chain of elements of $W^*$ contains only a finite number of $S^*$-transitions. 
Obviously $R^*$ is conversely well-founded because so is $R$. 
Hence, the relation $R^* \cup S^*$ is also conversely well-founded. 

It suffices to show that $A$ is not valid in the model $\mathcal{M}^*$. 

\begin{cl}
For any $\mathcal{L}(\rhd)$-formula $B$ and any $(x, n) \in W^*$, if $d(B) \leq n$, then 
\[
	(x, n) \Vdash^* B \iff x \Vdash B. 
\]
\end{cl}
\begin{proof}
We prove the claim by induction on the construction of $B$. 
We give only a proof of the case that $B$ is of the form $C \rhd D$. 
Assume $d(C \rhd D) \leq n$. 

$(\Rightarrow)$: 
Suppose $(x, n) \Vdash^* C \rhd D$. 
Let $y \in W$ be any element such that $x R y$ and $y \Vdash C$. 
Then, $(x, n) R^* (y, n)$. 
Since $d(C) \leq d(C \rhd D) \leq n$, by the induction hypothesis, $(y, n) \Vdash^* C$. 
Then, there exists a $(z, m) \in W^*$ such that $(y, n) S^* (z, m)$ and $(z, m) \Vdash^* D$. 
In this case, $y S z$ and $n = m + 1$. 
Since $d(D) + 1 \leq d(C \rhd D) \leq n = m + 1$, we have $d(D) \leq m$. 
Then, by the induction hypothesis, $z \Vdash D$. 
Therefore, $x \Vdash C \rhd D$. 

$(\Leftarrow)$: 
Suppose $x \Vdash C \rhd D$. 
Let $(y, m) \in W^*$ be such that $(x, n) R^* (y, m)$ and $(y, m) \Vdash^* C$. 
Then, $x R y$ and $n = m$. 
Since $d(C) \leq d(C \rhd D) \leq n = m$, by the induction hypothesis, $y \Vdash C$. 
Then, there exists a $z \in W$ such that $y S z$ and $z \Vdash D$. 
Since $d(D) + 1 \leq d(C \rhd D) \leq n$, we have $d(D) \leq n -1$. 
By the induction hypothesis, $(z, n - 1) \Vdash^* D$. 
Also, $(y, m) S^* (z, n - 1)$ because $m = (n - 1) + 1$. 
We conclude $(x, n) \Vdash^* C \rhd D$.~\qedhere
\end{proof}

By the claim, we obtain $(w, d(A)) \nVdash^* A$, and hence $A$ is not valid in $\mathcal{M}^*$. 
\end{proof}

Our proof of Theorem~\ref{AC} is merely a tracing of the proof of~\cite[Theorem 4.1]{Kur18_2}. 
Therefore, we only give an outline of a proof and we leave the details to that paper.  

From our proofs of Theorems~\ref{SVC} and~\ref{SVC2}, we obtain a primitive recursively represented simplified $\IL^-(\PP)$-model $\mathcal{M} = (W, R, S, \Vdash)$ satisfying the following conditions: 
\begin{itemize}
	\item $W = \omega$; 
	\item For any $x \in W \setminus \{0\}$, $0 R x$; 
	\item $S$ is transitive and $R \cup S$ is conversely well-founded; 
	\item The restriction of $\mathcal{M}$ to the set $W \setminus \{0\}$ is a disjoint union of infinitely many finite simplified $\IL^-(\PP)$-models; 
	\item For any $\mathcal{L}(\rhd)$-formula $A$, if $\IL^-(\PP) \nvdash A$, then there exists an $i \in W \setminus \{0\}$ such that $i \nVdash A$. 
\end{itemize}

Let $R^*$ be the transitive closure of $R \cup S$. 
Also, let $\sim$ be an equivalence relation on $W \setminus \{0\}$ defined by
\[
	i \sim j : \iff i\ \text{and}\ j\ \text{belong to the same model in the disjoint union}.
\]
Let $x R y$, $x S y$, $x R^* y$, and $x \sim y$ be $\Delta_1(\PA)$ formulas naturally representing the corresponding binary relations in $\PA$. 
We may assume that $\PA$ proves several basic facts about these relations. 
For example, for each $i \neq 0$, $\PA \vdash \forall x (\overline{i} R x \leftrightarrow \bigvee_{i R j} x = \overline{j})$. 

Let $\sigma(u)$ be a $\Sigma_1$ numeration of $T$ such that $T \nvdash \neg \Con_T^n$ for all $n \in \omega$, where $\Con_T^n$ is the sentence obtained by applying $\PR_T(\cdot)$ to $0=1$ $n$ times. 
The existence of such a numeration is proved in~\cite[Lemma 7]{Bek90}. 
Notice that the assumption that $T$ is an extension of $\PA$ is used to guarantee the existence of $\sigma(u)$. 

Since $R^*$ is transitive and conversely well-founded, as in the usual proof of the uniform arithmetical completeness theorem of $\GL$, we obtain a $\Sigma_2$ Solovay formula $\alpha(x)$ satisfying the following conditions (cf.~\cite[Chapter 9]{Boo93}): 
\begin{enumerate}
	\item $\PA \vdash \exists x \alpha(x)$;
	\item $\PA \vdash \forall x \forall y(\alpha(x) \land \alpha(y) \to x = y)$;
	\item $\PA \vdash \forall x \forall y(\alpha(x) \land x R^* y \to \neg \PR_{\sigma}(\gn{\neg \alpha(\dot{y})}))$; 
	\item $\PA \vdash \forall x(\alpha(x) \land x \neq 0 \to \PR_{\sigma}(\gn{\exists z(\dot{x} R^* z \land \alpha(z))}))$; 
	\item $\N \models \alpha(0)$. 
\end{enumerate}
Here $\N$ is the standard model of arithmetic. 

In the following, let $Q \in \{R, S\}$. 
Define $\delta_Q(x, u)$ to be the $\Sigma_1$ formula
\[
	\exists z(x \sim z \land u = \gn{\neg \alpha(\dot{z})} \land \neg x Q z). 
\]
Also, define $\gamma_Q(u)$ to be the $\Sigma_2$ formula
\[
	\exists x(\alpha(x) \land x \neq 0 \land \delta_Q(x, u)). 
\]
Let $\tau_Q(u)$ be the $\Sigma_2$ formula $\sigma(u) \lor \gamma_Q(u)$. 
Since $\PA \vdash \alpha(0) \to \forall u \neg \gamma_Q(u)$, we have $\PA \vdash \alpha(0) \to \forall u \bigl(\tau_Q(u) \leftrightarrow \sigma(u) \bigr)$. 
Then, $\tau_Q(u)$ is a $\Sigma_2$ numeration of $T$ (cf.~\cite[Lemma 4.13]{Kur18_2}). 

As in the proof of~\cite[Theorem 4.1]{Kur18_2}, we can prove the following lemmas. 

\begin{lem}[{cf.~\cite[Lemma 4.6]{Kur18_2}}]\label{Lem1}\leavevmode
\begin{enumerate}
	\item $\PA \vdash \forall x \forall y \bigl(\alpha(x) \land x \neq 0 \land x \sim y \land \neg x Q y \to \PR_{\tau_Q}(\gn{\neg \alpha(\dot{y})}) \bigr)$. 
	\item $\PA \vdash \forall x \bigl(\alpha(x) \land x \neq 0 \to \PR_{\tau_Q}(\gn{\exists z (\dot{x} Q z \land \alpha(z))}) \bigr)$. 
\end{enumerate}
\end{lem}

\begin{lem}[{cf.~\cite[Lemma 4.10]{Kur18_2}}]\label{Lem2}
If $i Q j$, then $\PA \vdash \alpha(\overline{i}) \to \neg \PR_{\tau_Q}(\gn{\neg \alpha(\overline{j})})$. 
\end{lem}

\begin{lem}[{cf.~\cite[Lemmas 4.3, 4.7, and 4.8]{Kur18_2}}]\label{Lem3}\leavevmode
\begin{enumerate}
	\item $\PA \vdash \forall x \forall y \bigl(x \neq 0 \land x Q y \to (\delta_Q(x, u) \to \delta_Q(y, u)) \bigr)$. 
	\item $\PA \vdash \exists x \exists z \bigl(x \neq 0 \land x Q z \land \alpha(z) \land \PR_{\sigma \lor \delta_Q(x)}(\gn{\varphi}) \bigr) \to \PR_{\tau_Q}(\gn{\varphi})$ for any $\mathcal{L}_A$-formula $\varphi$. 
	\item $\PA \vdash \PR_{\tau_Q}(\gn{\varphi}) \to \PR_{\tau_Q}(\gn{\PR_{\tau_Q}(\gn{\varphi})})$ for any $\mathcal{L}_A$-formula $\varphi$. 
\end{enumerate}
\end{lem}

Lemma~\ref{Lem3}.3 means that $\tau_Q(u)$ satisfies the L\"ob condition. 
Notice that the transitivity of $Q$ is used in the proof of Lemma~\ref{Lem3}.1. 
Let $f$ be the arithmetical interpretation defined by $f(p) \equiv \exists x (\alpha(x) \land x \Vdash p)$. 
Finally, we prove the following lemma. 

\begin{lem}\label{Lem4}
Let $B$ be any $\mathcal{L}(\rhd$)-formula and $i \neq 0$. 
\begin{enumerate}
	\item If $i \Vdash B$, then $\PA \vdash \alpha(\overline{i}) \to \FRS{B}$; 
	\item If $i \nVdash B$, then $\PA \vdash \alpha(\overline{i}) \to \neg \FRS{B}$. 
\end{enumerate}
\end{lem}	
\begin{proof}
We prove Clauses 1 and 2 simultaneously by induction on the construction of $B$. 
We give only a proof of the case that $B$ is of the form $C \rhd D$. 

1. Suppose $i \Vdash C \rhd D$. 
Let $j \in W$ be such that $i R j$. 
If $j \nVdash C$, then by the induction hypothesis, $\PA \vdash \alpha(\overline{j}) \to \neg \FRS{C}$. 
If $j \Vdash C$, then there exists a $k \in W$ such that $j S k$ and $k \Vdash D$. 
By the induction hypothesis, $\PA \vdash \alpha(\overline{k}) \to \FRS{D}$. 
Then, $\PA \vdash \neg \PR_{\tau_S}(\gn{\neg \alpha(\overline{k})}) \to \Con_{\tau_S + \FRS{D}}$. 
By Lemma~\ref{Lem2}, we get $\PA \vdash \alpha(\overline{j}) \to \neg \PR_{\tau_S}(\gn{\neg \alpha(\overline{k})})$, and hence $\PA \vdash \alpha(\overline{j}) \to \Con_{\tau_S + \FRS{D}}$. 

We have proved that 
\[
	\PA \vdash \forall x \bigl(\overline{i} R x \land \alpha(x) \to (\FRS{C} \to \Con_{\tau_S + \FRS{D}}) \bigr).
\]
Then, 
\[
	\PA \vdash \PR_{\tau_R}(\gn{\exists x(\overline{i} R x \land \alpha(x))}) \to \PR_{\tau_R}(\gn{\FRS{C} \to \Con_{\tau_S + \FRS{D}}}).
\]
Since $\PA \vdash \alpha(\overline{i}) \to \PR_{\tau_R}(\gn{\exists x(\overline{i} R x \land \alpha(x))})$ by Lemma~\ref{Lem1}.2, we obtain
\[
	\PA \vdash \alpha(\overline{i}) \to \PR_{\tau_R}(\gn{\FRS{C} \to \Con_{\tau_S + \FRS{D}}}).
\]
This means $\PA \vdash \alpha(\overline{i}) \to \FRS{C \rhd D}$. 

2. Suppose $i \nVdash C \rhd D$. 
Then, there exists a $j \in W$ such that $i R j$, $j \Vdash C$ and for all $k \in W$ with $j S k$, $k \nVdash D$. 
Then, by the induction hypothesis, $\PA$ proves $\alpha(\overline{j}) \to \FRS{C}$ and $\forall x \bigl(\overline{j} S x \land \alpha(x) \to \neg \FRS{D} \bigr)$. 
Then, $\PA \vdash \PR_{\tau_S}(\gn{\exists x(\overline{j} S x \land \alpha(x))}) \to \PR_{\tau_S}(\gn{\neg \FRS{D}})$. 
Since $\PA \vdash \alpha(\overline{j}) \to \PR_{\tau_S}(\gn{\exists x(\overline{j} S x \land \alpha(x))})$ by Lemma~\ref{Lem1}.2, we have $\PA \vdash \alpha(\overline{j}) \to \PR_{\tau_S}(\gn{\neg \FRS{D}})$. 
We have $\PA \vdash \alpha(\overline{j}) \to \FRS{C} \land \PR_{\tau_S}(\gn{\neg \FRS{D}})$. 
Then, we obtain
\[
	\PA \vdash \neg \PR_{\tau_R}(\gn{\neg \alpha(\overline{j})}) \to \neg \PR_{\tau_R}(\gn{\FRS{C} \to \Con_{\tau_S + \FRS{D}}}).
\]
Since $\PA \vdash \alpha(\overline{i}) \to \neg \PR_{\tau_R}(\gn{\neg \alpha(\overline{j})})$ by Lemma~\ref{Lem2}, we get
\[
	\PA \vdash \alpha(\overline{i}) \to \neg \PR_{\tau_R}(\gn{\FRS{C} \to \Con_{\tau_S + \FRS{D}}}).
\]
We conclude $\PA \vdash \alpha(\overline{i}) \to \neg \FRS{C \rhd D}$. 
\end{proof}

\begin{proof}[Proof of Theorem~\ref{AC}]
Suppose $\IL^-(\PP) \nvdash A$. 
Then, there exists an $i \neq 0$ such that $i \nVdash A$. 
By Lemma~\ref{Lem4}, $\PA \vdash \alpha(\overline{i}) \to \neg \FRS{A}$. 
Then, $\PA \vdash \neg \PR_{\sigma}(\gn{\neg \alpha(\overline{i})}) \to \neg \PR_{\sigma}(\gn{\FRS{A}})$. 
Since $0 R^* i$, $\PA \vdash \alpha(0) \to \neg \PR_{\sigma}(\gn{\neg \alpha(\overline{i})})$, and hence $\PA \vdash \alpha(0) \to \neg \PR_{\sigma}(\gn{\FRS{A}})$. 
Since $\N \models \alpha(0)$, we have $\N \models \neg \PR_{\sigma}(\gn{\FRS{A}})$. 
Therefore, we conclude $T \nvdash \FRS{A}$. 
\end{proof}

\section{Concluding remarks}

In the present paper, we have analyzed properties of the logic $\IL^-(\PP)$ from various perspectives. 
In particular, we found that $\IL^-(\PP)$ satisfies some logically good properties that $\IL^-$ does not. 
For example, $\IL^-$ does not have CIP and $\ell$FPP, whereas $\IL^-(\PP)$ does have these properties (Theorems~\ref{CIP} and \ref{lFPP}). 
Also, for the systems of sequent calculi $\ILms$ and $\ILmPs$ that we introduced, the cut-elimination theorem does not hold for the former, but it does for the latter (Theorem \ref{CE} and Corollary \ref{FCE}).
We also proved that $\IL^-(\PP)$ is a base logic for some relational semantics, that is, $\IL^-(\PP)$ is characterized by the class of all simplified $\IL^-(\PP)$-frames (Theorem \ref{SVC}).
Furthermore, we also proved that $\IL^-(\PP)$ has a connection with arithmetic, that is, it is complete with respect to some suitable arithmetical semantics (Theorem \ref{AC}). 
Therefore, we can say that the persistence principle $\PP$ is a well behaved principle for $\IL^-$. 

On the other hand, $\IL^-(\PP)$ does not have FPP, and UFP does not hold for $\IL^-(\PP)$, so it cannot be said that the logic can fully express the properties of arithmetic (Corollary \ref{FailureofFPP2}).
In this sense, it may be expected in the future to study meaningful extensions of $\IL^-(\PP)$. 
One candidate is $\IL^-(\J{4}_+, \PP)$, which is a sublogic of $\ILP$ and enjoys FPP.
Furthermore, we would like to mention as a candidate another logic that seems to make sense arithmetically. 
In our arithmetical semantics for $\IL^-(\PP)$, $\Box$ is interpreted by a provability predicate $\PR_{\tau_0}(\cdot)$, which is not necessarily $\Sigma_1$. 
If we restrict our argument to interpreting $\Box$ by a $\Sigma_1$ provability predicate, then $\mathcal{L}(\rhd)$-formulas that are not contained in $\IL^-(\PP)$ would become arithmetically valid. 
For example, the following Ignatiev's axiom $\mathbf{Sa}$ is such a formula: 
\begin{description}
	\item [Sa] $A \rhd B \to \bigl(A \land (C \rhd D) \bigr) \rhd \bigl(B \land (C \rhd D) \bigr)$. 
\end{description}
We propose to study the logic $\IL^-(\PP, \mathbf{Sa})$ following the work of the present paper. 

In Section~\ref{Sec_SV}, we proved that $\IL^-(\PP)$ is faithfully embeddable into bimodal logics $\GLK$ and $\FGL$ (Proposition \ref{Embed}). 
Inspired from this result, we introduced arithmetical semantics for $\IL^-(\PP)$ and directly proved that $\IL^-(\PP)$ is arithmetically complete. 
From this situation, we conjecture that the bimodal logics $\GLK$ and $\FGL$ are complete for the corresponding arithmetic semantics. 

\begin{prob}\leavevmode
\begin{enumerate}
	\item Is the logic $\GLK$ complete with respect to the arithmetical semantics based on pairs $(\tau_0, \tau_1)$ of numerals such that $\tau_0$ satisfies the L\"ob condition? 
	\item Is the logic $\FGL$ complete with respect to the arithmetical semantics based on pairs $(\tau_0, \tau_1)$ of numerals such that both $\tau_0$ and $\tau_1$ satisfy the L\"ob condition? 
\end{enumerate}
\end{prob}

From the previous studies on extensions of $\IL^-$, we know that two different logics $\IL^-(\J{2}_+, \J{5})$ and $\IL^-(\J{4}_+, \PP)$ have FPP. 
Then, we propose the following problem. 

\begin{prob}
Does the intersection of $\IL^-(\J{2}_+, \J{5})$ and $\IL^-(\J{4}_+, \PP)$ enjoy FPP?
\end{prob}

\section*{Acknowledgement}

The second author was supported by JSPS KAKENHI Grant Number JP19K14586. 
We would like to thank the anonymous referees for their careful reading of an earlier version of the paper.

\bibliographystyle{plain}
\bibliography{ref}

\end{document}